\documentclass[11pt]{amsart}

\newcommand{\cal}[1]{\mathcal{#1}}

\newtheorem*{theoremA}{Theorem A}
\newtheorem*{theoremB}{Theorem B}
\newtheorem*{corollaryA1}{Corollary A1}
\newtheorem*{corollaryA2}{Corollary A2}
\usepackage{john}
\usepackage{hyperref}
\usepackage{fullpage}

\usepackage[displaymath, mathlines]{lineno}

\numberwithin{example}{subsection} 
\theoremstyle{definition}
\newtheorem{example-special}[theorem]{Example} 
\numberwithin{equation}{subsection}

\usepackage[all,cmtip]{xy}

\DeclareMathOperator{\seed}{seed}

\title{Arithmetic properties of Fredholm series \\for $p$-adic modular forms}
\author{John Bergdall and Robert Pollack}
\date{\today}

\address{John Bergdall\\Department of Mathematics and Statistics \\ Boston University \\ 111 Cummington Mall \\ Boston, MA 02215\\USA}
\email{bergdall@math.bu.edu}
\urladdr{http://math.bu.edu/people/bergdall}

\address{Robert Pollack\\Department of Mathematics and Statistics \\ Boston University \\ 111 Cummington Mall \\ Boston, MA 02215\\USA}
\email{rpollack@math.bu.edu}
\urladdr{http://math.bu.edu/people/rpollack}


\begin{document}
\setcounter{tocdepth}{1}

\begin{abstract}
We study the relationship between recent conjectures on slopes of overconvergent $p$-adic modular forms ``near the boundary'' of $p$-adic weight space. We also prove in tame level 1 that the coefficients of the Fredholm series of the $U_p$ operator never vanish modulo $p$, a phenomenon that fails at higher level.  In higher level, we do check that infinitely many coefficients are non-zero modulo $p$ using a modular interpretation of the mod $p$ reduction of the Fredholm series recently discovered by Andreatta, Iovita and Pilloni.
\end{abstract}

\maketitle
\tableofcontents

\section{Introduction}
A recent preprint of Wan, Xiao and Zhang contains a conjecture (see \cite[Conjecture 2.5]{WanXiaoZhang-Slopes}) which makes precise a folklore possibility, inspired by a theorem of Buzzard and Kilford \cite{BuzzardKilford-2adc}, regarding the  slopes of $p$-adic modular forms near the ``boundary of weight space''.

 The conjecture comes in two parts. First, the slopes of overconvergent $p$-adic eigenforms at weights approaching the boundary should change linearly with respect to the valuation of the weight. The second part of the conjecture is that these slopes, after normalizing by this conjectured linear change, form a finite union of arithmetic progressions. The main goal in this paper is to prove that the second conjecture is a consequence of the first.

Implicit in the conjecture is a simple and beautiful description of these slopes:\ near the boundary the slopes arise from a scaling of the Newton polygon of the mod $p$ reduction of the Fredholm series of the $U_p$ operator. The first half of this paper is devoted to studying this characteristic $p$ object using a $p$-adic version of the trace formula discovered by Koike in the 1970s \cite{Koike-padicProperties}. 
Our main result in this direction is that this mod $p$ Fredholm series is not a polynomial, i.e.\ it is a true power series with infinitely many non-zero coefficients. This observation in turn is an important step in the deduction of the main theorem.

Regarding the broader context, work of Andreatta, Iovita and Pilloni (see \cite{AndreattaIovitaPilloni-Halo})  have brought to light an extraordinary theory, envisioned by Robert Coleman, of overconvergent $p$-adic modular forms {\em in characteristic $p$}! The connection with the characteristic zero theory is a modular interpretation of the mod $p$ reduction of the Fredholm series of $U_p$. Our results on this mod $p$ reduction show that there are infinitely many finite slope eigenforms in characteristic $p$ and our methods give a concrete way to understand this new characteristic $p$ theory. 

\subsection{}
We begin setting notation now.
We  fix a prime $p$ and an integer $N \geq 1$, called the tame level,
 such that $(N,p) = 1$. If $k \in \N$ is an integer and $\Gamma \subseteq \SL_2(\Z)$ is a congruence subgroup then we denote by $M_{k}(\Gamma)$ the space of classical modular forms of weight $k$ and level $\Gamma$. Similarly, $S_{k}(\Gamma)$ denotes the space of cusp forms.  We also let $S_k(\Gamma_1(M),\chi)$ denote the space of cusp forms of level $\Gamma_1(M)$ with character $\chi :(\Z/M\Z)^\times \to \C^\times$.

\subsection{}
Choose an embedding $\bar \Q \inject \bar \Q_p$ and use this to make a choice of a $p$-adic valuation on $\bar \Q$ satisfying $v_p(p) = 1$. We also choose an isomorphism $\C \simeq \bar \Q_p$. Using this isomorphism we view spaces of modular forms as vector spaces over $\bar \Q_p$.

\subsection{}
Let $\Delta \ci \Z_p^\x$ be the multiplicative torsion subgroup and write $\Z_p^\times \cong \Delta \times \Gamma$ where $\Gamma = 1+p\Z_p$ if $p$ is odd and $\Gamma = 1+4\Z_2$ if $p=2$. 
Let $A$ be an affinoid $\Q_p$-algebra. The $p$-adic weight space $\cal W$ is defined on $A$-points by 
\begin{equation*} 
\cal W(A) = \Hom_{\cont}(\Z_p^\x, A^\x).
\end{equation*}
Weight space is a union of open discs. Explicitly, we can give a coordinate $w(\kappa)$ on $\cal W$ by $w(\kappa) := \kappa(\gamma)-1$ where $\gamma \in \Gamma$ is some fixed topological generator.  The coordinate $w(\kappa)$ does not determine $\kappa$ as it does not depend on $\restrict{\kappa}{\Delta}$.  
But, we have a natural isomorphism
\begin{align}
\cal W(A) &\iso \Hom(\Delta, A^\x) \x \set{z \in A \st v_A^{\sup}(z) > 0}\label{eqn:weight-space}\\
\kappa &\mapsto (\restrict{\kappa}{\Delta}, w(\kappa)),\nonumber
\end{align}
where $v_A^{\sup}(-)$ denotes the valuation corresponding to the supremum semi-norm on $A$.\footnote{The set of $z \in A$ for which $v_A^{\sup}(z) > 0$ is the set of topologically nilpotent elements in $A$. Equivalently, it is the set of $a \in A$ such that $v_p(a(x)) > 0$ for all $x \in \Sp(A)$ \cite[Proposition 6.2.3/2]{BGR}.} Denote by $\hat{\Delta}$ the group of characters $\Delta \goto \C_p^\times$. Thus we can write
\begin{equation*}
\cal W = \bigunion_{\eta \in \hat{\Delta}} \cal W_{\eta}
\end{equation*}
where $\cal W_\eta = \set{\kappa \st \restrict{\kappa}{\Delta} = \eta}$ is an open $p$-adic unit disc. We say that $\cal W_{\eta}$ is an even component if $\eta$ is an even character, i.e.\ $\eta(-1) = 1$. Finally, note that the isomorphism \eqref{eqn:weight-space} depends on the choice of $\gamma$, but $v_p(w(\kappa))$ does not.
 
\subsection{} For each $\kappa \in \cal W(\C_p)$, Coleman has defined in \cite{Coleman-pAdicBanachSpaces} spaces of overconvergent $p$-adic modular forms $M_\kappa^\dagger(N)$ of tame level $\Gamma_0(N)$ (note that we suppress the choice of $p$ in the notation). There are also close cousins $S_\kappa^{\dagger}(N)$ of cusp forms.  Each of $M_\kappa^{\dagger}(N)$ and $S_\kappa^{\dagger}(N)$ is an $LB$-space (a compact inductive limit of Banach spaces) and there is the usual Hecke action where, in particular, $U_p$ acts compactly.  These spaces are trivial unless $\kappa$ is even because we work in level $\Gamma_0(N)$.

To be more precise, if $v > 0$ is a rational number then one has the notion of $v$-overconvergent $p$-adic modular forms $M_{\kappa}(N)(v)$ of tame level $\Gamma_0(N)$, which form a $p$-adic Banach space (using the notation from \cite[Section B.2]{Coleman-ClassicalandOverconvergent}). If $0 < v' < v$ then there is an injective transition map $M_{\kappa}(N)(v) \inject M_{\kappa}(N)(v')$ which is compact and $M_{\kappa}^{\dagger}(N) = \dirlim_{v>0} M_\kappa(N)(v)$ is their directed limit. For $v$ sufficiently small, the $U_p$-operator is a continuous operator $M_{\kappa}(N)(v) \goto M_\kappa(N)(pv)$ which then defines a compact endomorphism of $M_\kappa(N)(v)$. The characteristic series of $U_p$ acting on $M_\kappa(N)(v)$ is independent of $0 < v$ sufficiently small and we write $\det(1 - t\restrict{U_p}{M_{\kappa}^{\dagger}(N)})$ for this series. 

Most of this article will focus on the characteristic series for cusp forms $S_\kappa^{\dagger}(N)$ and thus we make the notation
\begin{equation*}
P(\kappa,t) = \det(1 - t\restrict{U_p}{S_\kappa^{\dagger}(N)}) \in \Q_p(\kappa)[[t]].
\end{equation*}
The series $P(\kappa,t)$ differs from $\det(1 - t\restrict{U_p}{M_\kappa^{\dagger}(N)})$ only by a factor corresponding to ordinary Eisenstein families. For example, if $N = 1$ then this factor is simply $1-t$.

\subsection{}
By Coleman's work \cite{Coleman-pAdicBanachSpaces}, $P(\kappa,t)$ is analytic in $\kappa$ and in fact the coefficients of $P(\kappa,t)$ are defined by power series over $\Z_p$ in $\kappa$ (we explain this in Theorem \ref{thm:koike-iwasawa} below). Explicitly, for each component $\cal W_\eta \ci \cal W$, there exists power series $a_{i,\eta}(w) \in \Z_p[\![w]\!]$ such that
\begin{equation*}
\kappa \in \cal W_\eta \implies P(\kappa,t) = 1 + \sum_{i=1}^\infty a_{i,\eta}(w(\kappa))t^i.
\end{equation*}

\subsection{}
If $k \in \N$ then we may consider the character $(z\mapsto z^k) \in \cal W(\Q_p)$. There is a canonical embedding $M_k(\Gamma_0(Np)) \inject M_{z^k}^{\dagger}(N)$ which is equivariant for the corresponding Hecke action. Slightly more generally, if $\chi: (\Z/p^t\Z)^\x \goto \bar \Q_p^\x$ is a primitive Dirichlet character and $k \in \Z$ there is also a Hecke equivariant embedding 
$M_{k}(\Gamma_1(Np^t),\chi) \inject M_{z^k\chi}^{\dagger}(N)$.

We refer to the images of these embeddings, as you run over all $k$ and all $\chi$, as the classical subspaces. The weights $\set{z^k}$ and $\set{z^k\chi}$ are called the algebraic and locally algebraic weights. Their union is referred to as the set of arithmetic weights. It is worth pointing out the following regarding the valuations of arithmetic weights:
\begin{lemma}\label{lemma:weight-absolute-values}
Let $k \in \Z$ and $\chi$ be a primitive Dirichlet character modulo $p^t$.
\begin{enumerate}
\item $v_2(w(z^k)) = 2 + v_2(k)$ and if $p > 2$ then $v_p(w(z^k)) = 1 + v_p(k)$.
\item If $t \geq 3$ then $v_2(w(z^k\chi)) = {1 \over 2^{t-3}}$.
\item If $p > 2$ and $t \geq 2$ then $v_p(w(z^k\chi)) = {1 \over \phi(p^{t-1})} = {1 \over p^{t-2}(p-1)}$.
\end{enumerate}
\end{lemma}
\begin{proof}
The first computation is an immediate application of the binomial theorem. The proof of (b) is just as the proof of (c), so now let's assume that $p$ is an odd prime. Note that $\gamma \in \Gamma$ is a generator for the kernel of the reduction map $(\Z/p^t\Z)^\times \goto \Delta \simeq (\Z/p)^\times$. In particular, $\chi(\gamma) = \zeta_{p^{t-1}}$ is a primitive $p^{t-1}$st root of unity. Then, since $\gamma^k \equiv 1 \pmod{p}$ and $v_p(\zeta_{p^{t-1}} - 1) < 1$, we get that
\begin{equation*}
v_p(\gamma^k \chi(\gamma) -  1) = v_p(\zeta_{p^{t-1}} - 1) = {1 \over \phi(p^{t-1})}
\end{equation*}
as $t \geq 2$.  This concludes the proof.
\end{proof}
\subsection{}
Lemma \ref{lemma:weight-absolute-values} explains that arithmetic weights live in two separate regions of weight space: 
\begin{itemize}
\item If $p$ is odd then the algebraic weights $z^k$ all live in the ``center region'' $v_p(-) \geq 1$. If $p=2$ and $k$ is even  then the algebraic weights $z^k$ live in the region $v_2(-) \geq 3$. 
\item If $p$ is odd then the locally algebraic weights of conductor at least $p^2$ live in an ``outer region'' $v_p(-) \leq {1 \over p-1} < 1$. If $p=2$ then the locally algebraic weights of conductor at least $8$ lie in the region $v_2(-) \leq 1 < 3$.
\end{itemize}


\subsection{}
Fix a component $\cal W_\eta \ci \cal W$. Since $a_{i,\eta}$ is a power series over $\Z_p$, the Weierstrass preparation theorem implies that, if $a_{i,\eta}\neq 0$, then we can write
$a_{i,\eta} = p^{\mu}f(w)u(w)$
where 
\begin{itemize}
\item $f(w) = w^{\lambda} + \dotsb \in \Z_p[w]$ is a monic polynomial of degree $\lambda \geq 0$ which is a monomial modulo $p$,
\item $\mu$ is a non-negative integer, and 
\item $u(w)$ is a unit in $\Z_p[\![w]\!]$.
\end{itemize}
We note that $\lambda$ is the number of zeroes of $a_{i,\eta}(w)$ in the open unit disc $v_p(-) > 0$. Furthermore, $\mu$ and $\lambda$ do not depend on the choice of $\gamma$ and neither does $v_p(a_{i,\eta}(w_0))$ for a fixed $w_0$ in $\cal W_\eta$. In particular, the following questions are completely independent of the choice of $\gamma$.
\begin{itemize}
\item What is $\mu(a_{i,\eta})$?
\item What is $\lambda(a_{i,\eta})$ and what are the slopes of the zeroes of $a_{i,\eta}$?
\item What is the Newton polygon of $P(w_0,t)$, for a fixed weight $w_0$?
\end{itemize}

For the remainder of the introduction, we focus on the first and the last of these questions.  See Section \ref{sec:examples} for the middle question.

\subsection{}
For $\kappa \in \cal W$, write $\nu_1(\kappa) \leq \nu_2(\kappa) \leq \dotsb$ for the slopes of the Newton polygon of $P(\kappa,t)$ (or, equivalently, the slopes of $U_p$ acting on $S_\kappa^{\dagger}(N)$).  The following is a reformulation of \cite[Conjecture 2.5]{WanXiaoZhang-Slopes} (see also \cite{LiuXiaoWan-IntegralEigencurves}).

\begin{conjecture}\label{conj:slope-conjecture}
For each component $\cal W_\eta$, there exists an $r > 0$ such that:
\begin{enumerate}
\item For $\kappa \in \cal W_\eta$, the Newton polygon of $P(\kappa,t)$ depends only on $v_p(w(\kappa))$ if $0 < v_p(w(\kappa)) < r$.  Moreover, on the region $0 < v_p(-) < r$, the indices of the break points of the Newton polygon of $P(\kappa,t)$ are independent of $\kappa$.
\item If $i$ is the index of a break point in the region $0 < v_p(-) < r$ then $\mu(a_{i,\eta}) = 0$.
\end{enumerate}
\begin{enumerate}
\item[(c)] The sequence $\set{\nu_i(\kappa)/v_p(w(\kappa))}$ is a finite union of arithmetic progressions, independent of $\kappa$, if $0 < v_p(w(\kappa)) < r$ and $\kappa \in \cal W_\eta$.
\end{enumerate}
\end{conjecture}

\subsection{}
Conjecture \ref{conj:slope-conjecture} is known completely in only two cases: when $p=2$ or $p=3$ and $N =1$. The case $p=2$ is due to Buzzard and Kilford \cite{BuzzardKilford-2adc}. The case $p=3$ is due to Roe \cite{Roe-Slopes}. In the case where either $p=5$ or $p=7$ and $N=1$, Kilford \cite{Kilford-5Slopes} and Kilford and McMurdy \cite{KilfordMcMurday-7adicslopes} verified part (c) for a single weight.

More recently, Liu, Wan and Xiao have proven the analogous conjecture in the setting of overconvergent $p$-adic modular forms for a definite quaternion algebra \cite[Theorems 1.3 and 1.5]{LiuXiaoWan-IntegralEigencurves}. By the Jacquet--Langlands correspondence, their results give striking evidence for, and progress towards, Conjecture \ref{conj:slope-conjecture}. 
\subsection{}
Let's reframe Conjecture \ref{conj:slope-conjecture} in terms which hint at the approaches we've mentioned being developed separately by Andreatta, Iovita and Pilloni \cite{AndreattaIovitaPilloni-Halo} and Liu, Wan, and Xiao \cite{LiuXiaoWan-IntegralEigencurves}.

On a fixed component $\cal W_{\eta}$, write $P(w,t) := P_\eta(w,t) = 1+\sum a_{i,\eta}(w) t^i \in \Z_p[\![w,t]\!]$ for the two-variable characteristic power series of $U_p$ and  $\bar{P(w,t)} \in \F_p[\![w,t]\!]$ for its mod $p$ reduction. Viewing $\bar{P(w,t)}$ as a one-variable power series in $t$ over $\F_p[\![w]\!]$, it has a Newton polygon which we call the $w$-adic Newton polygon of $\bar{P(w,t)}$.  

We can rephrase Conjecture \ref{conj:slope-conjecture}(a,b) in terms of the Newton polygon of $\bar{P(w,t)}$.

\begin{conjecture}\label{conj:mod-p-reduction}
For each component $\cal W_{\eta}$, there exists an $r > 0$ such that if $\kappa \in \cal W_{\eta}$ and $0 < v_p(w(\kappa)) < r$ then the Newton polygon of $P(w(\kappa),t)$ equals the $w$-adic Newton polygon of $\bar{P(w,t)}$ scaled by $v_p(w(\kappa))$.\footnote{By the scaling of a Newton polygon, we mean the vertical scaling of all points on the Newton polygon.}
\end{conjecture}

The equivalence of Conjecture \ref{conj:mod-p-reduction} and Conjecture \ref{conj:slope-conjecture}(a,b) is proven in Proposition \ref{prop:equiv}. The proof makes key use of Corollary A2 below.

\begin{remark}
In \cite{AndreattaIovitaPilloni-Halo}, Andreatta, Iovita and Pilloni give a {\em modular} interpretation of $\bar{P(w,t)}$ as the characteristic power series of a compact operator acting on an $\F_p(\!(w)\!)$-Banach space.
\end{remark}

\subsection{}

We now state our first result which implies Conjecture \ref{conj:slope-conjecture}(b) in the case of tame level 1.

\begin{theoremA}[Theorem \ref{thm:mu=0-intext}]
If $N = 1$ and $\eta$ is even then $\mu(a_{i,\eta}) = 0$ for each $i \geq 1$.
\end{theoremA} 
Our proof of Theorem A (see Section \ref{sec:mu=0}) makes use of an older tool:\ Koike's computation of the traces of Hecke operators via a $p$-adic limit of the Eichler--Selberg trace formula \cite{Koike-padicProperties}.  

The assumption that $N = 1$ is necessary for Theorem A  (compare with Example \ref{example:p2N3}). However, we can also show that the $\mu$-invariants of the $a_i$ vanish infinitely often in higher tame level (see Corollary A2).  We thank Vincent Pilloni for suggesting the corollaries that follow.

\subsection{}
In  \cite{AndreattaIovitaPilloni-Halo}, Andreatta, Iovita and Pilloni construct, for each component $\cal W_{\eta}$, an $\F_p(\!(w)\!)$-Banach space $\bar M_{\bar \kappa_\eta}^{\dagger}(N)$ of ``overconvergent $p$-adic modular forms in characteristic $p$'' equipped with a compact operator $\bar U_p$. By \cite[Corollaire 1.1]{AndreattaIovitaPilloni-Halo} the characteristic  power series of $\bar U_p$ acting on $\bar M_{\bar \kappa_{\eta}}^{\dagger}(N)$ is equal to $\bar{P(w,t)}$. Thus an immediate corollary of Theorem A is:
\begin{corollaryA1}
For all tame levels $N$ and every even component $\cal W_{\eta}$, there exists infinitely many finite slope eigenforms for $\bar U_p$ in $\bar M_{\bar \kappa_{\eta}}^{\dagger}(N)$.
\end{corollaryA1}
\begin{proof}
If $N = 1$, Theorem A implies that the characteristic power series $\det(1 - \restrict{t\bar U_p}{\bar M_{\bar \kappa_\eta}^{\dagger}(1)}) = \bar{ P(w,t)}$ is not a polynomial. Thus, the infinitely many roots of $\bar{P(w,t)}$ correspond to infinitely many finite slope eigenform for $\bar U_p$ in $\bar M_{\bar \kappa_\eta}^{\dagger}(1)$.

For $N > 1$, the theory in \cite{AndreattaIovitaPilloni-Halo} implies that the usual degeneracy maps between modular curves induce an injective $\bar U_p$-equivariant embedding $\bar M_{\bar \kappa_\eta}^{\dagger}(1) \inject \bar M_{\bar \kappa_\eta}^{\dagger}(N)$. Since $\bar M_{\bar \kappa_\eta}^{\dagger}(1)$ has infinitely many finite slope forms, the same is true for $\bar M_{\bar \kappa_\eta}^{\dagger}(N)$.
\end{proof}

\begin{remark}
Coleman deduced that there are infinitely many finite slope overconvergent $p$-adic eigenforms for $U_p$ acting in a classical weight $k$ using a similar argument (see \cite[Proposition I4]{Coleman-pAdicBanachSpaces}).
\end{remark}

\begin{corollaryA2}
For each fixed $N \geq 1$ and even $\eta$, we have $\mu(a_{i,\eta}) = 0$ for infinitely many $i$.
\end{corollaryA2}
\begin{proof}
This is equivalent to the statement that $\det(1 - \restrict{t\bar U_p}{\bar M_{\bar \kappa_\eta}^{\dagger}(N)})$ is not a polynomial which follows immediately from Corollary A1.
\end{proof}

\begin{remark}
One could also argue that characteristic power series in tame level 1 divides the one in tame level $N$ as entire functions over weight space.  The same would then follow for their mod $p$ reductions.
Since the mod $p$ reduction of the tame level 1 series is not a polynomial by Theorem A, the series in tame level $N$ is not a polynomial either.  This gives a direct argument for Corollary A2 allowing us to reverse the logic above and deduce Corollary A1 from Corollary A2.
\end{remark}

\subsection{}
Our second result identifies the essential part of Conjecture \ref{conj:slope-conjecture}.  We show that if the Newton polygons at weights near the boundary behave uniformly in the weight (i.e.\ Conjecture \ref{conj:slope-conjecture}(a)), then it is automatic that the $\mu$-invariants of the $a_i$ vanish whenever $i$ is breakpoint (Conjecture \ref{conj:slope-conjecture}(b)) and the slopes near the boundary form a finite union of arithmetic progressions (Conjecture \ref{conj:slope-conjecture}(c)).

\begin{theoremB}[Theorems \ref{thm:thmB(a)to(b)} and \ref{thm:arith-progs}]
If Conjecture \ref{conj:slope-conjecture}(a) holds on $0 < v_p(-) < r$ for every even component of weight space then the same is true for Conjecture \ref{conj:slope-conjecture}(b) and (c).
\end{theoremB}

The implication (a) implies (b) follows easily from Corollary A2 and is Theorem \ref{thm:thmB(a)to(b)} in the text.  The proof we give for (a) implies (c) (see Theorem \ref{thm:arith-progs}) again uses a classical tool (the Atkin-Lehner involutions on cuspforms with nebentype) and a slightly more modern one (Coleman's classicality theorem for overconvergent cuspforms). A similar argument was noticed independently by Liu, Wan and Xiao and used to complete the main result of their paper \cite[Section 4.2]{LiuXiaoWan-IntegralEigencurves}.  We remark that our proof of (a) implies (b) holds component-by-component.  However, our proof of (a) implies (c) truly requires as an input that Conjecture \ref{conj:slope-conjecture}(a) holds simultaneously  on  all components of weight space.

\subsection{}

One may also consider level $\Gamma_1(N)$. In that case, the results stated here in the introduction remain true. Namely, Theorem A is obviously true still and its corollaries are nearly formal. The reader may check that the proof of Theorem B carries over as well.

However if one fixes a tame character, or one works over odd components of $p$-adic weight space, then our method of proof does not generalize to handle Corollaries A1 and A2.  Indeed, the basic method we use begins by examining the subspace of level 1 forms which only contributes to forms with trivial nebentype on even components.  If one could prove Corollary A2 in this case, then our proof of Theorem B would carry over. We do though expect these results to hold in general.


\subsection*{Acknowledgements}
The authors would like to thank Kevin Buzzard for extraordinary interest, encouragement and comments on an early draft of this paper. We would also like to thank Liang Xiao for several discussions pertaining to the two papers \cite{WanXiaoZhang-Slopes, LiuXiaoWan-IntegralEigencurves}. We owe Vincent Pilloni for discussing the application(s) of Theorem A to the work \cite{AndreattaIovitaPilloni-Halo}. We duly thank Fabrizio Andreatta, Matthew Emerton, Adrian Iovita and Glenn Stevens for helpful conversations. Finally, we thank the anonymous referee for their helpful comments and suggestions. The first author was partially supported by NSF grant DMS-1402005 and the second author was partially supported by NSF grant DMS-1303302.

\section{The trace formula and $\mu$-invariants}\label{sec:mu=0}

In this section we will prove Theorem A (i.e.\ that $\mu(a_{i,\eta})=0$ in tame level 1).  The key to our proof is 
to explicitly write down the characteristic power series of $U_p$,
not with coefficients in $\Z_p[\![w]\!]$, but in a coordinate free manner over $\Z_p[\![\Gamma]\!]$.  For this we will use an explicit formula for the traces of Hecke operators \cite{Koike-padicProperties,  Hijikata-TraceFormula}.

\subsection{}
Consider the Iwasawa ring $\Z_p[\![\Z_p^\times]\!]$. If $\rho \in \Z_p^\times$ then the function $\kappa \mapsto \kappa(\rho)$ defines an analytic function $[\rho]$ on the weight space $\cal W$, and thus gives a 
natural inclusion $\Z_p[\![\Z_p^\times]\!] \inject \cal O(\cal W)$. Natural here refers to the canonical decompositions of $\Z_p[\![\Z_p^\times]\!]$ and $\cal O(\cal W)$ as we now explain.

Given a character $\eta: \Delta \goto \C_p^\times$, the group algebra $\Z_p[\Delta]$ comes equipped with a map $\Z_p[\Delta] \goto \Z_p$ which evaluates group-like elements on $\eta$. If $\rho \in \Z_p^\times$ let $\bar \rho$ be its image in $\Delta$ under the splitting $\Z_p^\times \iso \Delta \times \Gamma$. For $\eta \in \hat{\Delta}$ we write $[\rho]_{\eta} := \eta(\bar \rho)\cdot \left[\rho\cdot \omega(\rho)^{-1}\right] \in \Z_p[\![\Gamma]\!]$ where $\omega$ denotes the Teichm\"uller character.\footnote{We explicitly see $\omega$ as the composition $\Z_p^\times \surject \Delta \subset \Z_p^\times$, making $\Z_p^\times$ the domain of $\omega$ in order to make $\omega$ a $p$-adic weight. By construction one has $\omega(\bar \rho) = \omega(\rho)$ for all $\rho \in \Z_p^\times$ and thus $\rho\cdot \omega(\rho)^{-1} \in \Gamma$ for such $\rho$.}  Then we have a $\Z_p$-module embedding
\begin{align}
\Z_p[\![\Z_p^\times]\!] &\inject \bigdsum_{\eta} \Z_p[\![\Gamma]\!]\label{eqn:rootsunity}\\
[\rho] &\mapsto \left([\rho]_{\eta}\right)_{\eta \in \hat{\Delta}}\nonumber
\end{align}
This map is an isomorphism for $p > 2$, or after inverting $p$. Similarly, the components of weight space are also indexed by characters $\eta$, and the canonical embedding above preserves the components.  That is, we have a commuting diagram
\begin{equation*}
\xymatrix{
\Z_p[\![\Z_p^\times]\!] \ar@{^{(}->}[d]  \ar@{^{(}->}[r] & \bigdsum_{\eta} \Z_p[\![\Gamma]\!] \ar@{^{(}->}[d]\\
 \cal O(\cal W) \ar@{=}[r]& \bigdsum_{\eta} \cal O(\cal W_\eta).
 }
\end{equation*}

\subsection{}
It is well-known that there is an isomorphism of rings
\begin{align*}
\Z_p[\![\Gamma]\!] &\overto{\simeq} \Z_p[\![w]\!]\\
[\gamma] & \mapsto w+1.
\end{align*}
Thus we can define the $\mu$ and $\lambda$-invariants of elements of $\Z_p[\![\Gamma]\!]$ via this isomorphism. They do not depend on the choice of $\gamma$.
\begin{lemma}\label{lemma:mu}
If $f = \sum_g c_g [g] \in \Z_p[\![\Gamma]\!]$ with $c_g = 0$ for all but finitely many $g$, then $\mu(f) = 0$ if and only if there exists some $g \in \Gamma$ such that $c_g \in \Z_p^\x$.
\end{lemma}
\begin{proof}
One direction is clear, so let us prove that if some $c_g$ is a unit then $\mu(f) = 0$.
Let $[g_1],\dotsc,[g_r]$ be the terms of $f$ with $c_{g_i} \neq 0$. Choose an $n$ sufficiently large so that the images $\bar g_i$ of $g_i$ in $\Gamma/\Gamma^{p^n}$ are all distinct.  Then the image of $\bar f$ in the quotient $\Z_p[\Gamma/\Gamma^{p^n}]$ is given by $\sum c_{g_i}[\bar g_i]$. But the $[\bar g_i]$ are all distinct  and form a subset of a basis for the finite free $\Z_p$-module $\Z_p[\Gamma/\Gamma^{p^n}]$. 
In particular, if some $c_{g_i}$ is a unit, then $\bar f \not \congruent 0 \bmod p$, and hence $f \not \congruent 0 \bmod p$.
\end{proof}
\begin{remark}
The same argument shows that if $f\neq 0$ is as in the lemma then 
\begin{equation*}
\mu(f) = \min\set{n \st  c_g \congruent 0 \bmod p^n \text{ for all $g$}}.
\end{equation*}
\end{remark}

\subsection{}

If $s,j\geq 1$ are integers and $p \ndvd s$ then the quadratic polynomial $\Phi_{s,j}(X) := X^2 - sX + p^j$ factors (say over $\bar \Q$) as
\begin{equation*}
X^2 - sX + p^j = (X-\rho_{s,j})(X-\bar \rho_{s,j}).
\end{equation*}
Under our fixed embedding $\bar \Q \inject \bar \Q_p$, this polynomial has exactly one root which is a $p$-adic unit; we label this root by $\rho_{s,j}$.  Viewing this polynomial over $\C$,  its roots are complex conjugates when $1 \leq s < 2p^{j/2}$, in which case the complex absolute value of $\rho_{s,j}$ is $p^{j/2}$. We make a slightly more general observation in the following lemma.

\begin{lemma}\label{lemma:rhoe}
If $\bf e = \set{e(s,j)}$ is a finite collection of positive integers, depending on pairs $(s,j)$ such that $1 \leq s < 2p^{j/2}$, then $\rho_{\bf e} := \prod_{s,j} \rho_{s,j}^{e(s,j)}$ is not a root of unity.
\end{lemma}
\begin{proof}
We just explained that $\abs{\rho_{s,j}}_\infty = p^{j/2}> 1$ for any choice of complex absolute value $\abs{-}_{\infty}$. In particular, $\abs{\rho_{\bf e}}_{\infty} > 1$ as well, meaning that $\rho_{\bf e}$ cannot be a root of unity.
\end{proof}

\subsection{}
The quadratic surds $\rho_{s,j}$ and $\bar \rho_{s,j}$ classically play a role in explicit formulae for the trace of $T_{p^j}$ acting on spaces $S_k(\Gamma_0(N))$ of cuspidal modular forms. Here we recall what happens when one passes to the space of overconvergent cuspforms for $\Gamma_0(N)$.

\begin{theorem}[Koike's formula]\label{thm:koike}
Suppose that $N,j \geq 1$ are integers. Then there exists constants $c_N(p,j)$ and $c_N(p,s,j)$ in $\Z_p$ such that for each even integer $k$, 
\begin{equation*}
\tr(\restrict{U_p^j}{S_k^{\dagger}(N)}) = -c_N(p,j) - \sum_{\substack{1 \leq s < 2p^{j/2}\\ p \ndvd s}} {c_N(p,s,j)\over \rho_{s,j}^2 - p^j}\rho_{s,j}^k.
\end{equation*}
\end{theorem}
\begin{proof}
In the case $N=1$, this is \cite[Theorem 1]{Koike-padicProperties}.  
We quickly sketch the general argument since this is well-known to experts. A combination of Coleman's control theorem (see Theorem \ref{thm:coleman-control}(a) below), the theory of newforms, and the gluing of spaces of overconvergent cuspforms over weight space implies that for each integer $k$, 
\begin{equation}\label{eqn:limit}
\tr(\restrict{U_p^j}{S_{k}^{\dagger}(N)}) = \lim_{n\goto \infty} \tr(\restrict{T_{p^j}}{S_{k+p^n(p-1)}(\Gamma_0(N))}),
\end{equation}
the limit being a $p$-adic limit. For $n$ sufficiently large, $k+p^n(p-1) \geq 2$ and so one can evaluate the limit using the explicit trace formula Hijikata gives in \cite[Theorem 0.1]{Hijikata-TraceFormula}. The limit is easily computed from the description given by Hijikata once we note that the Hecke operators there are off by a factor of $p^{k/2-1}$ from ours (compare with \cite[Section 2.2]{Koike-padicProperties}).
\end{proof}
\begin{remark}
We will make explicit these constants in the case $N=1$ in the following section and in general in Appendix \ref{app:constants}.  For now, we just note that each $c_N(-)$ is a $p$-adically integral algebraic number which a computer can effectively compute.
\end{remark}

\subsection{}
For the remainder of Section \ref{sec:mu=0} we fix an even component $\cal W_{\eta}$. For each $j\geq 1$, we denote by $T_j \in \cal O(\cal W_{\eta})$ the function $\kappa \mapsto \tr(\restrict{U_p^j}{S_{\kappa}^{\dagger}(N)})$. Koike's formula shows that $T_j$ lies in the subring $\Z_p[\Gamma] \ci \Z_p[\![\Gamma]\!] \ci \cal O(\cal W_{\eta})$.

\begin{theorem}[Koike's formula II]\label{thm:koike-iwasawa}
If $j\geq 1$ then $T_j \in \Z_p[\Gamma]$ and is given by
\begin{equation*}
T_j = -c_N(p,j)[1]_{\eta} - \sum_{\substack{1 \leq s < 2p^{j/2}\\ p \ndvd s}} {c_N(p,s,j)\over \rho_{s,j}^2 - p^j}[\rho_{s,j}]_{\eta}
\end{equation*}
where the notation is the same as in Theorem \ref{thm:koike}.
\end{theorem}
\begin{proof}
When we evaluate the right-hand side of the given formula at $z^k \in \cal W_{\eta}$ we  get $T_j(z^k)$ by Theorem \ref{thm:koike}. Since the integers inside $\cal W_{\eta}$ accumulate on themselves, the formula follows from the analyticity of $T_j$.
\end{proof}

\subsection{}\label{subsec:koikeN=1}
In order to prove Theorem A, we need to make the constants explicit in Koike's formula when $N=1$.  The constants arise from the explicit form of the Eichler--Selberg trace formula for the Hecke operators $T_{p^j}$ on cuspforms for $\SL_2(\Z)$ given by Zagier in an appendix to Lang's book \cite{Lang-IntroModularForms} (see also the correction \cite{Zagier-Correction}).  

We denote by $H: \Z \goto \Q$ the Hurwitz class numbers defined by $H(0) = -1/12$, $H(n) = 0$ if $n < 0$ and $H(n)$ is the number of equivalence classes of positive definite binary quadratic forms of discriminant $-n$ counted with certain multiplicities (depending on $n \bmod 3$). Then, it is easy to compute from \cite[Theorem 2, p.\ 48]{Lang-IntroModularForms} and \eqref{eqn:limit} that the constants in Koike's formula are given by
\begin{align*}
c_1(p,j) &= 1,\\
c_1(p,s,j) &= {H(4p^j - s^2)}.
\end{align*}
Specializing Theorem \ref{thm:koike-iwasawa} to the case $N = 1$ we get
\begin{corollary}[$N=1$]\label{cor:Tjiwasawa}
If $j\geq 1$ then 
\begin{equation*}
T_j = -[1]_{\eta} - \sum_{\substack{1 \leq s < 2p^{j/2}\\p\ndvd s}} {H(4p^j - s^2) \over \rho_{s,j}^2 - p^j}[\rho_{s,j}]_{\eta}.
\end{equation*}
\end{corollary}

\subsection{}
The characteristic power series $1 + \sum_{i\geq 1} a_{i}(\kappa)t^i := \det(1 - \restrict{tU_p}{S_{\kappa}^{\dagger}(N)})$ of $U_p$ over the weights in $\cal W_{\eta}$ can be computed in terms of $T_{j}$. Indeed, \cite[Corollaire 3]{Serre-CompactOperators} implies that for each $\kappa \in \cal W_{\eta}$,
\begin{equation*}
1 + \sum_{i\geq 1} a_{i}(\kappa)t^i = \det(1 - \restrict{tU_p}{S_{\kappa}^{\dagger}(N)}) = \exp\left(-\sum_{j\geq 1} \tr(\restrict{U_p^j}{S_{\kappa}^{\dagger}(N)}){t^j\over j}\right).
\end{equation*}
Unwinding, we get the classical symmetric functions identities
\begin{equation}\label{eqn:newton-relations}
a_{0} = 1 \;\;\;\; \;\;\;\; \;\;\;\; a_{i} = -{1 \over i} \sum_{j=1}^{i} a_{i-j}T_{j}
\end{equation}
(Note: The $a_i$ written here are the $a_{i,\eta}$ from the introduction.)
\subsection{}
We're now ready to show the vanishing of the $\mu$-invariants for the coefficients of the characteristic power series in tame level $N=1$.

\begin{theorem}[$N=1$]\label{thm:mu=0-intext}
For each $i\geq 1$, $\mu(a_{i}) = 0$.
\end{theorem}
\begin{proof}
Obviously if we define $a_0 = 1$ and $a_i$ recursively as $i a_i = - \sum_{j=1}^i a_{i-j}T_j$ then for reasons already stated we have $a_i \in \Z_p[\Gamma] \ci \cal O(\cal W_{\eta})$.

If $\bf e = \set{e(s,j)}$ is a collection of integers then we denote by $\rho_{\bf e} = \prod_{s,j} \rho_{s,j}^{e(s,j)} \in \Z_p^\times$, the same element as in Lemma \ref{lemma:rhoe}. We claim by induction on $i\geq 1$ that 
\begin{equation}\label{eqn:form}
a_i = [1]_{\eta} + \sum_{\rho_{\bf e}\neq 1} c_{\bf e}[\rho_{\bf e}]_{\eta} \in\Z_p[\Gamma]
\end{equation}
where $c_{\bf e} \neq 0$ for all but finitely many $\bf e$. If we prove this claim then the proof of the theorem is complete. Indeed, if $\rho_{\mathbf e} \neq 1$ then $\rho_{\mathbf e}/\omega(\rho_{\mathbf e}) \neq 1$ by Lemma \ref{lemma:rhoe}, and so $[\rho_{\mathbf e}]_{\eta}$ is not a scalar multiple of $[1]_{\eta}$. In particular, $\mu(a_i) = 0$ by Lemma \ref{lemma:mu}.

Let's prove the claim. For $i=1$ we have $a_1 = -T_1$ so the claim follows from Corollary \ref{cor:Tjiwasawa}. Suppose that $i > 1$. By Lemma \ref{lemma:rhoe}, if either $\rho_{\bf e}\neq 1$ or $\rho_{\bf e'}\neq 1$ then $\rho_{\bf e}\rho_{\bf e'} \neq 1$. Thus, by induction on $j=1,\dotsc,i-1$ and Corollary \ref{cor:Tjiwasawa} we see that
\begin{equation*}
a_{i-j}T_j = \left([1]_{\eta} + \sum_{\rho_{\bf e}\neq1} c_{\bf e}[\rho_{\bf e}]_{\eta}\right)\left(-[1]_{\eta} + \sum_{\rho_{\bf f}\neq 1} d_{\bf f}[\rho_{\bf f}]_{\eta} \right) = -[1]_{\eta} + \sum_{\rho_{\bf g}\neq 1} h_{\bf g}[\rho_{\bf g}]_{\eta}.
\end{equation*}
Using the recursive formula \eqref{eqn:newton-relations} for $ia_i$ we see that
\begin{equation*}
ia_i = i[1]_{\eta} + \text{ higher terms}
\end{equation*} 
and, dividing by $i$, this completes the proof of the claim.  (Note that we know {\em a priori} that the higher terms are all divisible by $i$ since $a_i$ is an Iwasawa function.)
\end{proof}
\begin{remark}
It's a bit of a miracle that the trace formula yields such uniform information. If instead we worked on the full space $M_{\kappa}^{\dagger}(1)$ of overconvergent modular functions then the unique ordinary Eisenstein family would erase the group-like element $[1]$ in the Koike's formula. Suddenly, the vanishing of the $\mu$-invariants is ``not obvious''  even for the functions $T_j$, let alone the $a_i$. Staying with cuspforms but working in level $N > 1$, the same kind of issue arises as $c_N(p,j)$ need not be 1 (compare with Lemma \ref{lemma:const-term-special}).
\end{remark}

\section{Boundary slopes and arithmetic progressions}\label{sec:boundary-slopes}
In this section we  prove Theorem B; that is, we  prove that Conjecture \ref{conj:slope-conjecture}(a) implies Conjecture \ref{conj:slope-conjecture}(b,c). Throughout this section we implicitly assume every component $\cal W_{\eta}$ of weight space is even (since we work with level $\Gamma_0(N)$).

\subsection{}
 For each component $\cal W_\eta \ci \cal W$, we have a characteristic power series for cuspforms
\begin{equation*}
P(w,t) = 1 + \sum_{i=1}^\infty a_i(w)t^i \in 1 + t\Z_p[\![w,t]\!].
\end{equation*}
Here we set $a_i = a_{i,\eta}$ for ease of notation.  
We replicate here Conjecture \ref{conj:slope-conjecture}(a) which we are going to assume for all of Section \ref{sec:boundary-slopes}.
\begin{conjecture}\label{conj:in-text}
There exists $r>0$ such that if $\kappa \in \cal W_\eta$ and $0 < v_p(w(\kappa)) < r$ then the Newton polygon of $P(\kappa,t)$ depends only on $v_p(w(\kappa))$.  Moreover, on this region the indices of the  break points of the Newton polygon of $P(\kappa,t)$ are independent of $\kappa$.
\end{conjecture}
Note that since $\cal W$ has finitely many components, Conjecture \ref{conj:in-text} is true for each component $\cal W_\eta$ if and only if it is true with one uniform $r > 0$ for all components. Note as well that the break points in Conjecture \ref{conj:in-text} may well depend on the component
but do not depend on the choice of $r$ (for which the conjecture is valid at least).

\subsection{}

In this section, using Corollary A2, we prove the first half of Theorem B.  Namely, we show that Conjecture \ref{conj:in-text} implies that $\mu(a_i) = 0$ whenever $i$ is one of the break points of the Newton polygon near the boundary.

We begin with a simple lemma on the $p$-adic valuation of values of Iwasawa functions.

\begin{lemma}
\label{lemma:vals}
Let $f \in \Z_p[\![\Gamma]\!]$ be non-zero, and assume all of the zeroes of $f$ in the open unit disc have valuation at least $r$.  Then for $z \in \C_p$ with $0 < v_p(z)<r$, we have 
$$
v_p(f(z)) = \mu(f) + \lambda(f) \cdot v_p(z).
$$
\end{lemma}

\begin{proof}
Let $w_1, w_2, \dots, w_{\lambda(f)}$ denote the roots of $f$.  We have
\begin{align*}
v_p(f(z)) &= \mu(f) + \sum_{j=1}^{\lambda(f)} v_p(z - w_j) 
 = \mu(f) + \sum_{j=1}^{\lambda(f)} v_p(z) 
 = \mu(f) + \lambda(f) \cdot v_p(z)
 \end{align*}
 where the second equality follows since $v_p(z) < r \leq  v_p(w_j)$ for each $j$.
 \end{proof}

\begin{theorem}\label{thm:thmB(a)to(b)}
Conjecture \ref{conj:in-text} implies Conjecture \ref{conj:slope-conjecture}(b).
\end{theorem}
\begin{proof}
Let $I = \set{i_1 < i_2 < \dotsb}$ denote the break points of the Newton polygon of $P(w_0,t)$ over the region $0 < v_p(w_0) < r$. 
Formally set $i_0 = 0$ and $a_0 = 1$. We prove by induction on $k\geq 0$ that $\mu(a_{i_k}) = 0$. The case where $k = 0$ is by choice of $a_0$. 

Now suppose that $k \geq 1$. By Corollary A2, we may choose an integer $i$ such that $i \geq i_k$ and $\mu(a_{i}) = 0$. Since $i_k$ is a break point of the Newton polygon, we know that
\begin{equation}\label{eqn:using-bound}
{v_p(a_{i_k}(w_0)) - v_p(a_{i_{k-1}}(w_0)) \over i_k - i_{k-1}} \leq {v_p(a_{i}(w_0)) - v_p(a_{i_{k-1}}(w_0)) \over i - i_{k-1}}
\end{equation}
for all $0 < v_p(w_0) < r$. Choose a rational number $r' \leq r$ such that all three functions $a_{i_{k-1}}, a_{i_k}, a_i$ are non-vanishing on $0 < v_p(w_0) < r'$. Then, if $0 < v_p(w_0) < r'$, Lemma \ref{lemma:vals} implies
\begin{equation*}
{\mu(a_{i_k}) \over i_k - i_{k-1}} + {\left(\lambda(a_{i_k}) - \lambda(a_{i_{k-1}})\right)v_p(w_0) \over i_k - i_{k-1}} \leq {\left(\lambda(a_{i}) - \lambda(a_{i_{k-1}})\right)v_p(w_0) \over i - i_{k-1}}.
\end{equation*}
Here we are using the vanishes of the $\mu$-invariants of $a_i$ and $a_{i_{k-1}}$ (which we know by induction).
In particular, taking $v_p(w_0) \goto 0$, we see that $\mu(a_{i_k}) = 0$ as desired.
\end{proof}
%
%

\subsection{}
We recall that in the introduction we reformulated Conjecture \ref{conj:in-text} in terms of the $w$-adic Newton polygon of $\overline{P(w,t)}$.  In this section, we verify the equivalence of this reformulation.

\begin{proposition}
\label{prop:equiv}
Conjecture \ref{conj:in-text} is equivalent to Conjecture \ref{conj:mod-p-reduction}.
\end{proposition}

\begin{proof}
It is clear that Conjecture \ref{conj:mod-p-reduction} implies Conjecture \ref{conj:in-text}.  
Conversely, suppose that Conjecture \ref{conj:in-text} is true. Write $I = \set{i_1 < i_2 < \dotsb}$ for the break points of the Newton polygon on $0 < v_p(w_0) < r$ and $\lambda_i := \lambda(a_i)$. By Theorem \ref{thm:thmB(a)to(b)}, $\mu(a_{i_j}) = 0$ for each $j$. Since $i_j$ is the index of a break point at each weight $0 < v_p(w_0) < r$, clearly $a_{i_j}(w_0) \neq 0$ at such $w_0$. So, by Lemma \ref{lemma:vals} the break points of the Newton polygon of $P(w_0,t)$ are given by 
\begin{equation*}
(i_j, v_p(a_{i_j}(w_0))) = (i_j, \lambda_{i_j}v_p(w_0))
\end{equation*}
at each $0 < v_p(w_0) < r$. Thus, to prove the proposition we need to show that the break points of the $w$-adic Newton polygon of $\bar{P(w,t)}$ are $\set{(i_j,\lambda_{i_j}) \st j = 1,2,\dotsc}$. 

Set $i_0 = 0$ and $\lambda_{i_0} = 0$. We claim that if $j\geq 0$, $i_j < i$ and $\mu(a_i) = 0$ then 
\begin{equation}\label{eqn:want}
{\lambda_{i_{j+1}} - \lambda_{i_j} \over i_{j+1}-i_j} \overset{?}{\leq} {\lambda_{i} - \lambda_{i_j} \over i-i_j}
\end{equation}
with a {\em strict} inequality if $i_{j+1} < i$. Given this claim, the proposition follows by the definition of the Newton polygon. (Note that $\mu(a_i)$ must vanish if the point $(i,\lambda_i)$ actually lies on the $w$-adic Newton polygon.)

Let's show the claim. Fix $j\geq 0$ and $i > i_j$ such that $\mu(a_i) = 0$. Choose $0 < r' < r$ such that the roots of the three functions $a_i(w), a_{i_j}(w)$ and $a_{i_{j+1}}(w)$ all  lie in the region $v_p(-) \geq r'$. Now choose any weight $w_0$ such that $0 < v_p(w_0) < r'$. Since $\set{i_1,i_2,\dotsc}$ are the break points of the Newton polygon of $P(w_0,t)$ we know, by definition, that
\begin{equation}\label{eqn:going_to_use}
 {v_p(a_{i_{j+1}}(w_0)) - v_p(a_{i_j}(w_0)) \over i_{j+1} - i_j} \leq {v_p(a_{i}(w_0)) - v_p(a_{i_j}(w_0)) \over i- i_j}
\end{equation}
with a strict inequality if $i_{j+1} < i$.  Then, we observe that Lemma \ref{lemma:vals} and Theorem \ref{thm:thmB(a)to(b)} imply that $v_p(a_\ell(w_0)) = \lambda_\ell \cdot v_p(w_0)$ for $\ell \in \set{i,i_j,i_{j+1}}$. Thus the inequality \eqref{eqn:going_to_use} is the inequality \eqref{eqn:want} scaled by $v_p(w_0)$. This completes the proof.
\end{proof}

\subsection{}
We now move on to the arithmetic properties of the slopes and prove the second half of Theorem B; namely,  Conjecture \ref{conj:in-text} implies $\set{{\nu_i(\kappa) \over v_p(w(\kappa))} \st i = 1,2,\dotsc}$  is a finite union of arithmetic progressions independent of $\kappa$.  
The argument we give intertwines the various components of weight space.  
For this reason, we will assume $r$ is small enough so that it is witness to Conjecture 3.1 on every component simultaneously.

\begin{theorem}\label{thm:arith-progs}
Assume Conjecture \ref{conj:in-text} holds for every component and choose an $r$ which witnesses the conjecture on every component simultaneously. 
For a fixed component $\cal W_\eta$,  the sequence 
\begin{equation*}
\set{{\nu_i(\kappa) \over v_p(w(\kappa))} \st i = 1,2,\dotsc}
\end{equation*}
is a finite union of arithmetic progressions independent of $\kappa$ if
 $0 < v_p(w(\kappa)) < r$ with $\kappa \in \cal W_\eta$.
 \end{theorem}

\begin{remark}
Assuming Conjecture \ref{conj:in-text}, we give an explicit description of these arithmetic progressions in terms of slopes of $U_p$ acting on various spaces of classical cuspforms of weight two (see Theorem \ref{thm:arith-progressions}).  Examples are given in Sections \ref{subsec:exp=2N=1.5.12.20015.2.26pm}--\ref{subsec:examplep=11N=1}.
\end{remark}

\subsection{}
We begin by checking that the sequence in Theorem \ref{thm:arith-progs} is independent of $\kappa$. 

\begin{proposition}
\label{prop:sufficient}
Assume Conjecture \ref{conj:in-text} holds on $0 < v_p(-) < r$ for a fixed component $\cal W_\eta$.  The sequence 
\begin{equation*}
\set{{\nu_i(\kappa) \over v_p(w(\kappa))} \st i = 1,2,3,\dotsc}.
\end{equation*}
is independent of $\kappa$ if $0 < v_p(w(\kappa)) < r$ and $\kappa \in \cal W_\eta$.
\end{proposition}

\begin{proof}
By  Proposition \ref{prop:equiv},  Conjecture \ref{conj:mod-p-reduction} holds (as it is equivalent to Conjecture \ref{conj:in-text}).  But Conjecture \ref{conj:mod-p-reduction} implies that the Newton polygon 
scaled by $1/v_p(w(\kappa))$ is independent of $\kappa$ if $0 < v_p(w(\kappa)) < r$ and $\kappa \in \cal W_\eta$.  In particular, the sequence of this proposition is independent of such $\kappa$.
\end{proof}

\begin{remark}
It is easy to see that $\nu_i(\kappa)/v_p(w(\kappa))$ is not necessarily an {\em integer}. For example, if $p=2$ and $N=3$ then the sequence begins ${1\over 2}, {1\over 2}, 1, 1,\dotsc$ (conjecturally, see Section \ref{subsec:examplep=2N=3}).
\end{remark}

\subsection{}
Given the previous remark, let us be specific about what we mean by a finite union of arithmetic progressions.
\begin{definition}
A sequence $\set{x_n \st n = 1,2,\dotsc}$ of rational numbers is called an arithmetic progression if there exists a rational number $x$ and an integer $m\geq 1$ such that $x_i = x + mi$ for for all $i \geq 1$.
\end{definition}

For example, ${1 \over 2}, 1, {3\over 2}, 2, \dotsc$ is the union of two arithmetic progressions under our definition.  (We admit one could take other definitions allowing $m$ to be anything). On the other hand, it is the scaling of $1,2,3,\dotsc$ by the rational ${1 \over 2}$. We leave the following elementary lemma for the reader, but note that the previous examples shows it becomes false if we remove the words ``finite union of'' from the statement.

\begin{lemma}\label{lemm:simple-lemma}
Let $y\neq 0$ be a rational number. Then a sequence $\set{x_n \st n=1,2,\dotsc}$ is a finite union of arithmetic progressions if and only if $\set{y x_n \st n = 1,2,\dotsc}$ is a finite union of arithmetic progressions.
\end{lemma}

\subsection{}
We set some helpful notation on slopes of modular forms.

\begin{definition}
Suppose that $\kappa \in \cal W$ and $X \ci [0,\infty)$ is a set. 
\begin{itemize}
\item We denote by $\nu_{\kappa}^{X}$ the multiset of slopes appearing in $S_{\kappa}^{\dagger}(N)$ which are also in $X$, recorded with multiplicities.  We write $\nu_{\kappa}^{\dag}$ for $\nu_{\kappa}^{[0,\infty)}$.
\item If $\kappa = z^k \chi$ is an arithmetic weight of conductor $p^t$ then we write  $\nu^{\cl,X}_{\kappa}$ as the set of slopes appearing in $S_k(\Gamma_1(Np^t),\chi)$ and contained in $X$, again counted with multiplicity. We write $\nu_{\kappa}^{\cl}$ for $\nu_{\kappa}^{\cl,[0,\infty)}$.
\end{itemize}
\end{definition}
There are obvious operations we can do on multisets of rational numbers. If $\nu$ and $\nu'$ are two such sets we denote by $\nu \union \nu'$ their union as a multiset.  If $e \geq 1$ is an integer then we write $\nu^{\oplus e}$ for the $e$-fold union 
\begin{equation*}
\nu^{\oplus e} := \underlabel{e}{\nu \union \dotsb \union \nu}.
\end{equation*}
If $m$ is an integer then we also write
\begin{equation*}
m \pm \nu = \set{m \pm v \st v \in \nu}.
\end{equation*}

\subsection{}
We now recall two theorems on slopes of modular forms.  Let $t\geq 1$ be an integer and $\chi: (\Z/p^t\Z)^\x \goto \bar \Q_p^\x$ be a primitive Dirichlet character.
\begin{proposition}\label{prop:atkin-lehner}
There exists an involution $w_{p^t}: S_k(\Gamma_1(Np^t),\chi) \goto S_k(\Gamma_1(Np^t),\chi^{-1})$ such that if $f$ is a eigenform then $w_{p^t}(f)$ is also an eigenform whose $U_p$ eigenvalue is given by $p^{k-1}a_p(f)^{-1}$.
\end{proposition}
\begin{proof}
The operator is usually described, up to a scalar depending on $k$, as the slash action of a certain matrix (see \cite[Theorem 4.6.16]{Miyake-ModularForms}). Adelically one considers the corresponding automorphic representations $\pi(f)$ and $\pi(f)\tensor \chi^{-1}$. One checks using the theory of the new vector that the level at $p$ is preserved. The computation of the Hecke eigensystems goes back to Casselman \cite[Section 3]{Casselman-Assortment}.
\end{proof}
The involution $w_{p^t}$ is often called the Atkin--Lehner involution.

\begin{corollary}\label{cor:atkin-lehner}
We have  $\nu^{\cl}_{z^k \chi} = k-1 - \nu^{\cl}_{z^k \chi^{-1}}$.
\end{corollary}

\subsection{}
We now state the relationship between $\nu_{\kappa}^{\dagger}$ and $\nu_{\kappa}^{\cl}$. In small slopes, the relationship is given by  Coleman's control theorem. We also need the more delicate ``boundary case'' in Coleman's work.  In what follows, we let $e(N)$ denote the number of  $p$-ordinary Eisenstein families for $\Gamma_0(N)$ or equivalently the number of cusps of $X_0(N)$.
 
\begin{theorem}\label{thm:coleman-control}
Suppose that $\kappa = z^k \chi$ is an arithmetic weight. Then:
\begin{enumerate}
\item $\nu_{z^k\chi}^{[0,k-1)} = \nu_{z^k\chi}^{\cl,[0,k-1)}$.
\item If $\kappa \neq z^2$ and $\chi(-1) = (-1)^k$ then $\nu_{z^k\chi}^{\set{k-1}} = \set{k-1}^{\oplus e(N)} \union (k-1 + \nu_{z^{2-k}\chi}^{\set{0}}) \union \nu_{z^k\chi}^{\cl,\set{k-1}}$.
\end{enumerate}
\end{theorem}
\begin{proof}
Both (a) and (b) are proved at the same time by using Coleman's $\theta$-operator. Namely, if $p^t$ denotes the conductor of $\chi$ then for each rational number $0 \leq \nu \leq k-1$ there is a linear map $\theta^{k-1}:M^{\dagger}_{z^{2-k}\chi}(N)^{\nu-(k-1)} \goto S_{z^k\chi}^{\dagger}(N)^{\nu}$ (the superscripts mean the ``slope $\nu$'' part). Moreover, $\ker(\theta^{k-1}) = (0)$ except if $\kappa = z^2$ (which we excluded in (b)) and $\coker(\theta^{k-1})$ may be identified with $S_k(\Gamma_1(Np^{\max(t,1)}),\chi)^{\nu}$. 
When $\chi$ is trivial then this is due to Coleman in \cite[Sections 6 and 7]{Coleman-ClassicalandOverconvergent}. The non-critical slope case $\nu < k-1$ with non-trivial character is also due to Coleman \cite[Theorem 1.1]{Coleman-OverconvergentModularFormsOfHigherLevel}. The modifications for the critical slope $\nu = k-1$ case and non-trivial $\chi$ follow from \cite[Proposition 2.5]{Bergdall-CompanionPoints} (after specifying the $p$-part of the nebentypus everywhere in the displayed sequence).

With the $\theta$-operator in hand, part (a) now follows because if $\nu < k-1$ then $M^{\dagger}_{z^{2-k}\chi}(N)^{\nu-(k-1)} \subset M^{\dagger}_{z^{2-k}\chi}(N)^{\nu<0} = (0)$. And part (b) follows because $M^{\dagger}_{z^{2-k}\chi}(N)^{0}$ is spanned by $S^{\dagger}_{z^{2-k}\chi}(N)^{0}$ together with the $e(N)$-many ordinary $p$-adic Eisenstein series for $\Gamma_0(N)$ (under the assumption that $\chi(-1) = (-1)^k = (-1)^{2-k}$ so the spaces are non-zero).
\end{proof}

\subsection{}

Our strategy to prove Theorem \ref{thm:arith-progs} is to verify it for a single weight $\kappa$ in each component of weight space in the region $0 < v_p(-) < r$.  (This suffices by Proposition \ref{prop:sufficient}.)  We note that, by Lemma \ref{lemma:weight-absolute-values}, for each $\cal W_\eta$, we can find arithmetic weights of the form $\kappa = z^2\chi$ with $\chi$ finite order and $0 < v_p(w(\kappa)) < r$.

\begin{theorem}\label{thm:arith-progressions}
Assume Conjecture \ref{conj:in-text} holds for every component and choose an $r<1$ which witnesses the conjecture on every component simultaneously. 
Fix a component $\cal W_\eta$, and choose $\kappa = z^2\chi \in \cal W_\eta$ with $\chi$ finite order and $0 < v_p(w(\kappa)) < r$.

The set $\nu_{z^2\chi}^{\dag}$ is a finite union of arithmetic progressions. More specifically, if 
$$
\nu_{\eta,\seed} := \set{1,\dots,\frac{|\Delta|}{2}}^{\oplus e(N)} \union \bigunion_{j = 0}^{p-3 \over 2}  \left( j + \nu_{z^2 \chi \omega^{-2j}}^{\cl}\right),
$$
then
$$
\nu^{\dag}_{z^2\chi} = \bigunion_{i=0}^\infty \left(\nu_{\eta,\seed} + i \cdot \frac{|\Delta|}{2}\right).
$$
\end{theorem}

\begin{remark}
Since $r < 1$ in Theorem \ref{thm:arith-progressions}, the conductor of $\chi$ is at least $p^2$ if $p$ is odd, or at least $16$ if $p=2$. Thus by Lemma \ref{lemma:weight-absolute-values}, we have $v_p(w(z^2\chi)) = v_p(w(z^k\chi))$ for any $k \in \Z$.  Thus, by Conjecture \ref{conj:in-text}, the above theorem holds for weights of the form $z^k \chi$ for any $k \in \Z$.  
\end{remark}

\begin{remark}
By Proposition \ref{prop:sufficient}, Theorem \ref{thm:arith-progressions} implies Theorem \ref{thm:arith-progs}.
\end{remark}

\begin{remark}
The prediction of the slopes in the theorem was guessed by Wan, Xiao and Zhang \cite[Remark 2.7]{WanXiaoZhang-Slopes}.  As mentioned in the introduction, an argument similar to the one we are about to give was noticed independently by Liu, Wan and Xiao in the sequel \cite{LiuXiaoWan-IntegralEigencurves} to \cite{WanXiaoZhang-Slopes}.
\end{remark}
\begin{remark}
An amusing feature of the end result is that the slopes in the component containing $z^2\chi$ are naturally generated by the slopes in {\em other components} of weight space. Is there a symmetry giving rise to this phenomena?
\end{remark}

The proof of Theorem \ref{thm:arith-progressions} will appear in the next paragraph. But first, Theorem \ref{thm:arith-progressions} expresses the slopes of overconvergent $p$-adic cuspforms of a fixed weight $z^2\chi$ (with $\chi$ even and sufficiently ramified) as a union of arithmetic progressions with common difference $\frac{|\Delta|}{2}$.  The following corollary, whose proof we leave to the reader, shows that if we union together these slopes over all components of weight space, the arithmetic progressions mesh nicely together and form a finite union of arithmetic progressions with common difference 1.  

\begin{corollary}\label{cor:arith-progressions}
Assume Conjecture \ref{conj:in-text} holds for every component and choose an $r<1$ which witnesses the conjecture on every component simultaneously. 
Choose $\kappa = z^2\chi \in \cal W$ with $\chi$ finite order, even and $0 < v_p(w(\kappa)) < r$.
Set $\displaystyle \nu_{\seed} := \{1\}^{\oplus \frac{|\Delta| e(N)}{2}} \union \bigunion_{\eta} \nu^{\cl}_{z^2 \chi \eta}$
and
$\displaystyle \nu^{\dag} := \bigunion_{\eta} \nu^{\dag}_{z^2 \chi \eta}$.
Then
$$
\nu^{\dag} = \bigunion_{i=0}^\infty \left( \nu_{\seed} + i \right).
$$
\end{corollary}

\subsection{}
The technique we use to prove Theorem \ref{thm:arith-progressions} is to examine the slopes in finite intervals and then take their union and rearrange. We begin with the following proposition.

\begin{proposition}\label{prop:intermediate}
Under the assumptions and notation of Theorem \ref{thm:arith-progressions}, if $k > 2$ is an integer then
\begin{equation*}
\nu_{z^2\chi}^{[k-2,k-1)} = \set{k-2}^{\oplus e(N)} \union \left((k-2) - \nu_{z^2\chi^{-1}\omega^{2k-6}}^{\cl,\set{0}}\right) \union \left((k-1) - \nu_{z^2\chi^{-1}\omega^{2k-4}}^{\cl,(0,1]}\right).
\end{equation*}
\end{proposition}
\begin{proof}
We first note for later that $\chi$ is an even character since $z^2\chi \in \cal W_{\eta}$ and $\eta$ is assumed to be even. We also note that $z^2 \chi$ and $z^k \chi \omega^{2-k}$ are two arithmetic weights living in the same component of weight space and, by Lemma \ref{lemma:weight-absolute-values}, they live on the same rim within their weight disc (i.e.\ have the same valuation).  Thus, since we are assuming Conjecture \ref{conj:in-text}, we have $\nu_{z^2\chi}^{[k-2,k-1)} = \nu_{z^k \chi \omega^{2-k}}^{[k-2,k-1)}$.  By Coleman's control theorem, Theorem \ref{thm:coleman-control}(a), we have
\begin{equation*}
\nu_{z^2\chi}^{[k-2,k-1)} = \nu_{z^k \chi \omega^{2-k}}^{\cl,[k-2,k-1)}.
\end{equation*}
But now spaces of classical cuspforms with nebentypus have the Atkin-Lehner symmetries and so by Corollary \ref{cor:atkin-lehner} we deduce
\begin{equation}\label{eqn:save1}
\nu_{z^2\chi}^{[k-2,k-1)}=  \nu_{z^k \chi \omega^{2-k}}^{\cl,[k-2,k-1)} = (k-1) - \nu_{z^k\chi^{-1}\omega^{k-2}}^{\cl,(0,1]}
\end{equation}
(notice the careful switch of the ends of the interval). Thus it remains to compute the term on the right-hand side.

By Coleman's control theorem again, since $k > 2$, we can erase the classical bit from the last part:
\begin{equation*}
\nu_{z^k\chi^{-1}\omega^{k-2}}^{\cl,(0,1]} = \nu_{z^k\chi^{-1}\omega^{k-2}}^{(0,1]}.
\end{equation*}
But now we apply Conjecture \ref{conj:in-text} to the weight $z^k\chi^{-1}\omega^{k-2}$, which generally lives on a new component than the one we started with (which is why we need to assume the conjecture for all components at once). The two weights $z^k\chi^{-1}\omega^{k-2}$ and $z^2 \chi^{-1}\omega^{2k-4}$ live in the same component of $\cal W$, so by Conjecture \ref{conj:in-text}, we get the second equality in:
\begin{equation}\label{eqn:save2}
\nu_{z^k\chi^{-1}\omega^{k-2}}^{\cl,(0,1]} = \nu_{z^k\chi^{-1}\omega^{k-2}}^{(0,1]} = \nu_{z^2 \chi^{-1}\omega^{2k-4}}^{(0,1]}.
\end{equation}
And now we have to be a little careful, since slope one forms in a space of weight two overconvergent cuspforms are not necessarily classical. We do know from Coleman's control theorem that
\begin{equation}\label{eqn:save3}
 \nu_{z^2 \chi^{-1}\omega^{2k-4}}^{(0,1)} =  \nu_{z^2 \chi^{-1}\omega^{2k-4}}^{\cl, (0,1)}.
\end{equation}
On the other hand, the boundary case of Coleman's control theorem (Theorem \ref{thm:coleman-control}(b) -- note that here we are using that $\chi$ is an even character), applied to $k=2$, gives us 
\begin{equation}\label{eqn:save4}
 \nu_{z^2 \chi^{-1}\omega^{2k-4}}^{\set{1}} = \set{1}^{\oplus e(N)} \union \left(1 + \nu_{\chi^{-1}\omega^{2k-4}}^{\set{0}}\right) \union \left(\nu_{z^2 \chi^{-1}\omega^{2k-4}}^{\cl, \set{1}} \right).
\end{equation}
Finally, since $\chi^{-1}\omega^{2k-4}$ lies in the same component as the weight two point $z^2\chi^{-1}\omega^{2k-6}$, by Hida theory (or Conjecture \ref{conj:in-text}) and Theorem \ref{thm:coleman-control}(a), we have
\begin{equation}\label{eqn:save5}
1 + \nu_{\chi^{-1}\omega^{2k-4}}^{\set{0}} = 1 + \nu_{z^2\chi^{-1}\omega^{2k-6}}^{\set{0}} = 1 + \nu_{z^2\chi^{-1}\omega^{2k-6}}^{\cl, \set{0}},
\end{equation}
Putting it all together gives
\begin{align*}
&\nu_{z^2\chi}^{[k-2,k-1)}\\ &= (k-1) - \nu_{z^k\chi^{-1}\omega^{k-2}}^{\cl,(0,1]} & \text{(by \eqref{eqn:save1})}\\
&= (k-1) - \nu_{z^2 \chi^{-1}\omega^{2k-4}}^{(0,1]} & \text{(by \eqref{eqn:save2})}\\
&= (k-1) - \left(\nu_{z^2 \chi^{-1}\omega^{2k-4}}^{\cl, (0,1)} \union \set{1}^{\oplus e(N)} \union \left(1 + \nu_{\chi^{-1}\omega^{2k-4}}^{\set{0}}\right) \union \left(\nu_{z^2 \chi^{-1}\omega^{2k-4}}^{\cl, \set{1}} \right)\right)  & \text{(by \eqref{eqn:save3}, \eqref{eqn:save4})}\\
&= (k-1) -  \left(\set{1}^{\oplus e(N)} \union \left(1 + \nu_{z^2\chi^{-1}\omega^{2k-6}}^{\cl ,\set{0}}\right) \union \nu_{z^2\chi^{-1}\omega^{2k-4}}^{\cl,(0,1]}\right) & \text{(by \eqref{eqn:save5})}.
\end{align*}
We're done now after distributing the $k-1$ everywhere.
\end{proof}

\subsection{}
We're now ready to prove Theorem \ref{thm:arith-progressions} and thus Theorem \ref{thm:arith-progs}.
\begin{proof}[Proof of Theorem \ref{thm:arith-progressions}]
By Theorem \ref{thm:coleman-control}(a) and Corollary \ref{cor:atkin-lehner}, we have
\begin{equation*}
\nu_{z^2\chi}^{[0,1)}  = \nu_{z^2\chi}^{\cl, [0,1)} = 1 - \nu_{z^2\chi^{-1}}^{\cl, (0,1]}.
\end{equation*}
For higher slopes we get, using Proposition \ref{prop:intermediate}, that
\begin{align*}
\nu_{z^2\chi}^{[1,2)} &= \set{1}^{\oplus e(N)} \union \left(1 - \nu_{z^2\chi^{-1}}^{\cl,\set{0}}\right) \union \left( 2 - \nu_{z^2\chi^{-1}\omega^{2}}^{\cl,(0,1]}\right),\\
\nu_{z^2\chi}^{[2,3)} &= \set{2}^{\oplus e(N)} \union \left(2 - \nu_{z^2\chi^{-1}\omega^2}^{\cl,\set{0}}\right) \union \left(3 - \nu_{z^2\chi^{-1}\omega^{4}}^{\cl,(0,1]}\right),\\
\nu_{z^2\chi}^{[3,4)} &= \set{3}^{\oplus e(N)} \union \left(3 - \nu_{z^2\chi^{-1}\omega^4}^{\cl,\set{0}}\right) \union \left(4 - \nu_{z^2\chi^{-1}\omega^{6}}^{\cl,(0,1]}\right),\\
\vdots\;\; &= \;\;\vdots
\end{align*}
From this, we prove easily by induction that
\begin{equation*}
\nu_{z^2\chi}^{\dag} = \set{1,2,3,\dotsc}^{\oplus e(N)} \union \bigunion_{j\geq 1} \left(j - \nu_{z^2\chi^{-1}\omega^{2(j-1)}}^{\cl,[0,1]}\right).
\end{equation*}
By Corollary \ref{cor:atkin-lehner}, we have
$$
j - \nu_{z^2\chi^{-1}\omega^{2(j-1)}}^{\cl,[0,1]} = j - (1-\nu_{z^2\chi\omega^{2(1-j)}}^{\cl,[0,1]}) = j - 1 + \nu_{z^2\chi\omega^{2(1-j)}}^{\cl}
$$
so that
\begin{equation}\label{eqn:almost-there!}
\nu_{z^2\chi}^{\dag} 
= \set{1,2,3,\dotsc}^{\oplus e(N)} \union \bigunion_{j\geq 0} \left(j + \nu_{z^2\chi\omega^{-2j}}^{\cl}\right).
\end{equation}

If $p = 2$ then $\omega^2 = 1$, and we have proven the desired formula. Suppose now that $p$ is odd. Since $\omega$ has order $p-1$, if $j \congruent j' \bmod {p-1 \over 2}$ we get $\omega^{2j} = \omega^{2j'}$.
Thus we also see that 
\begin{equation*}
j \congruent j' \bmod {p-1\over 2} \implies\nu_{z^2\chi\omega^{-2j}}^{\cl} = \nu_{z^2\chi\omega^{-2j'}}^{\cl}.
\end{equation*}
Rewriting the large union in \eqref{eqn:almost-there!}, we find that
\begin{equation*}
\bigunion_{j\geq 0} \left(j + \nu_{z^2\chi\omega^{-2j}}^{\cl}\right) = \bigunion_{j = 0}^{p-3 \over 2} \bigunion_{i = 0}^{\infty} \left( (j + \nu_{z^2 \chi\omega^{-2j}}^{\cl}) + {p-1\over 2} \cdot i\right).
\end{equation*}
Since
\begin{equation}
\set{1,2,3,\dotsc}^{\oplus e(N)}
=
\bigunion_{i=0}^\infty \left( \set{1,\dots,\frac{p-1}{2}}^{\oplus e(N)} + \frac{p-1}{2} \cdot i \right),
\end{equation}
the theorem is proven.
\end{proof}


\subsection{}\label{subsec:exp=2N=1.5.12.20015.2.26pm}
If $p=2$ and $N=1$, then Buzzard and Kilford proved Conjecture \ref{conj:in-text} in \cite{BuzzardKilford-2adc} with $r = 3$. There exists a unique even character $\chi_8$ of conductor $8$ and $v_2(w(z^2\chi_8)) = 1$. Since $S_2(\Gamma_1(8),\chi_8) = (0)$ our recipe predicts that the slopes in $S_{z^2\chi_8}^{\dagger}$ are given by $1,2,3,\dotsc$ which is consistent with what was proven in \cite{BuzzardKilford-2adc}.

On the other hand, there are two even characters $\chi_{16}$ modulo 16. And for each one, there is a unique cusp form in $S_{2}(\Gamma_1(16),\chi_{16})$. The eigenvalue of $U_2$ is checked (e.g.\ in {\tt sage} \cite{sagemath}) to be $-1 \pm i$ (depending on the character) and thus we see that the unique classical weight two slope is ${1\over 2}$. Our prediction then is that the slopes are $1,2,3,\dotsc$ together with $0+{1\over 2}, 1+{1\over 2}, 2+{1\over 2},\dotsc = {1\over 2}, {3\over 2}, {5\over 2} \dotsc$. Again, this agrees with the results proven in \cite{BuzzardKilford-2adc} since $v_2(w(z^2\chi_{16})) = {1\over 2}$.

\subsection{}\label{subsec:examplep=2N=3}
Suppose that $p=2$ and $N=3$.  Choose $\chi_8$ as in the previous example. The space $S_2(\Gamma_1(24),\chi_8)$ has dimension two, and the characteristic polynomial of $U_2$ is given by $x^2 + 2 x + 2$, giving us slope ${1\over 2}$ with multiplicity two. There are two Eisenstein series on $\Gamma_0(3)$ and thus Theorem \ref{thm:arith-progressions} implies that that the slopes in weight $z^2\chi_8$ are, if Conjecture \ref{conj:in-text} is true for some $r > 1$, given by the list
\begin{equation*}
{1\over 2},{1\over 2}, 1,1,{3\over 2}, {3\over 2}, 2,2,\dotsc .
\end{equation*}
This is in line with experimental evidence; see Table \ref{app:p=2N=3}. (Compare with Example \ref{example:p=2N=3-bad-radius}.)

\subsection{}\label{subsec:examplep=11N=1}
Now we look at a more complicated example, to give a flavor of how many arithmetic progressions are predicted by Theorem \ref{thm:arith-progressions}. Suppose that $p=11$ and $N=1$. We expect Conjecture \ref{conj:in-text} is true with $r = 1$.  Choose $\chi$ to be even, have order 11 and have conductor 121 so that $v_2(w(z^2\chi)) = {1\over 10}$.  To generate the list of arithmetic progressions of slopes appearing in $S_{z^2\chi}^{\dagger}(1)$,   Theorem \ref{thm:arith-progressions} says we need to examine the slopes occurring in $S_2(\Gamma_1(11),\chi \omega^{-2j})$ for $j=0,\dots,4$.  The following table gives these slopes (all scaled by 10, computed in {\tt sage}).
\begin{center}
\begin{tabular}{c|l}
$j$ & $10$ times slopes in $S_2(\Gamma_1(11),\chi\omega^{-2j})$\\
\hline
0 & $0, 2, 3, 4, 5, 5, 6, 7, 8, 10$\\
1 & $1, 2, 3, 4, 4, 5, 6, 7, 9, 9$\\
2 & $1, 2, 3, 3, 4, 5, 7, 8, 8, 9$\\
3 & $1, 2, 2, 3, 5, 6, 7, 7, 8, 9$\\
4 & $1, 1, 3, 4, 5, 6, 6, 7, 8, 9$
\end{tabular}
\end{center}

Theorem \ref{thm:arith-progressions} then predicts that the slopes in $S^\dag_{z^2\chi}(1)$ (scaled by 10) are given by the 55 arithmetic progressions with common difference 50 and starting terms:\
 0, 2, 3, 4, 5, 5, 6, 7, 8, 10, 10, 
 11, 12, 13, 14, 14, 15, 16, 17, 19, 19, 20, 
 21, 22, 23, 23, 24, 25, 27, 28, 28, 29, 30, 
 31, 32, 32, 33, 35, 36, 37, 37, 38, 39, 40, 
 41, 41, 43, 44, 45, 46, 46, 47, 48, 49, 50.
Note that we've included in this list the contribution of the progression $10,20,30,\dots$ (which arises since $e(1)=1$) as 5 separate arithmetic progressions with common difference 50.

By Conjecture \ref{conj:in-text},  these slopes determine the slopes in $S_2(\Gamma_1(1331),\chi')$ where $\chi'$ is an even character with order $121$ and conductor $1331$.  As a check, we computed the slopes in this space using {\tt sage} and indeed they came out exactly as predicted.

To illustrate Corollary \ref{cor:arith-progressions}, this conjecture predicts that the slopes in $\nu^{\dag} = \bigunion_{j=0}^{4} \nu^{\dag}_{z^2 \chi \omega^{2j}}$ (scaled by 10) is a union of 55 arithmetic progressions with common difference 10 and starting terms: 
$$
\set{0} \union \set{1}^{\oplus 5} \union \set{2}^{\oplus 5} \union \set{3}^{\oplus 6} \union \set{4}^{\oplus 5} \union \set{5}^{\oplus 6} \union \set{6}^{\oplus 5} \union \set{7}^{\oplus 6} \union \set{8}^{\oplus 5} \union \set{9}^{\oplus 5} \union \set{10}^{\oplus 6}.
$$
(The lone seed slope zero corresponds to the (unique) $11$-adic cuspidal Hida family of tame level one on the component of weights $k \congruent 2 \bmod 10$.)

\section{Questions and examples}\label{sec:examples}

The constants in the trace formula, Theorem \ref{thm:koike}, are easily computed on a computer as rational numbers and thus we can write the coefficients of the characteristic power series as exact elements in $\Z_p[\Gamma]$. On the other hand, for the ease of reading off interesting phenomena, it is necessary to write the coefficients as power series $a_i(w)$ in a $p$-adic variable $w$. In order to do that, we have to make the choice of a topological generator $\gamma$ for $\Gamma$ (e.g.\ $\gamma = 1+p$ if $p$ is odd and $\gamma = 5$ if $p = 2$) and convert the Iwasawa elements to power series. In doing so, we have to compute $p$-adic logarithms and thus can only work up to some $(p^N,w^M)$-adic precision. All the following computations were done this way and the code is posted at \cite{Robwebsite}.

\subsection{}
Consider the coefficients $a_i(w)$ of the characteristic power series $P(w,t)$ over a fixed component. If Conjecture \ref{conj:in-text} is true then there exists a region $0 < v_p(-) < r$ such that the break points of the Newton polygon occur at integers $i$ so that $a_i \neq 0 \bmod p$ and the zeroes of $a_i(w)$ lie in the region $v_p(-) \geq r$. We do not know any examples, for any $p$ or $N$, that disprove Conjecture \ref{conj:in-text} for the value $r=1$. In particular we have no example where $i$ is the index of a break point of the $w$-adic Newton polygon of $\bar{P(w,t)}$ but $a_i$ has a zero in the region $v_p(-) < 1$. What about the non-break points of the Newton polygon?
\begin{question}\label{quest:4.1}
Are the zeroes of $a_i(w)$ uniformly bounded inside the disc $v_p(-) \geq 1$?
\end{question}
An affirmative answer to Question \ref{quest:4.1} combined with the vanishing of the $\mu$-invariants for $N = 1$ would (easily) imply Conjecture \ref{conj:in-text} is true for $N=1$ with  $r=1$. Unfortunately, we have to give a negative answer.
\begin{answer}
No.
\end{answer}
\begin{example}\label{example:bad-zeros}
Let $p = 23$ and $N = 1$. Using weight coordinate $w = \kappa(24) - 1$ on the component corresponding to weights $k \congruent 6 \bmod 22$, we computed the $w$-adic expansions
\begin{align*}
a_1(w) &= (18\cdot 23 + \dotsb) + (20 + \dotsb)w + \dotsb,\\
a_2(w) &= (11\cdot 23^2 + \dotsb) + (4\cdot 23 + \dotsb)w + (15 + \dotsb)w^2 + \dotsb,\\
a_3(w) &= (4\cdot 23^4 + \dotsb ) + (4\cdot 23^3 + \dotsb)w + (6\cdot 23^2  + \dotsb) w^2  +  (23 + \dotsb) w^3 +\\
&(13\cdot 23 + \dotsb) w^4 + (3 + \dotsb) w^5 + \dotsb.
\end{align*}
Thus $\lambda(a_3) = 5$ and the roots of $a_3$ have $23$-adic valuation(s) $1,1,1,{1 \over 2}, {1\over 2}$. This is the smallest example of a prime $p$ with $N=1$, so that one of $a_1,\dotsc,a_4$ had zeroes outside the disc $v_{p}(-) \geq 1$. The next example we found was for $p=53$, where the culprit was $a_3(w)$ on the component corresponding to $k \congruent 4 \bmod 52$.

Computing the coefficient $a_4$ (for $p=23$ on the component containing the weight 6) we have
\begin{multline*}
a_4(w) = (21\cdot 23^6 +\dotsb) + (16\cdot 23^5 + \dotsb)w + (11\cdot 23^4 + \dotsb)w^2 + (3\cdot 23^3 + \dotsb)w^3 + \\(2\cdot 23^2 + \dotsb)w^4 + (10\cdot 23 + \dotsb)w^5 + (21 + \dotsb)w^6 + \dotsb .
\end{multline*}
Thus $\lambda(a_{4}) = 6$ and the zeroes of $a_4$ all lie on the circle $v_{23}(-) = 1$. Since $\lambda(a_2) = 2$, the point $(3,\lambda(a_3))$ is not on the $w$-adic Newton polygon of the mod $23$ reduction. This means that the zeroes of $a_3(w)$ lying outside the disc $v_{23}(-)\geq 1$ are somehow irrelevant to Conjecture \ref{conj:in-text}.
\end{example}

\begin{example}\label{example:zeros15}
You can even find examples of zeroes of $a_1 = -\tr(U_p)$ outside of the central region $v_p(-)\geq 1$. Let $p=5$ and $N=3$. On the component $k \congruent 0 \bmod 4$ with weight coordinate $w = \kappa(6) - 1$, we found that
\begin{equation*}
a_1(w)  = (2\cdot 5 + \dotsb) + (5 + \dotsb ) w + (2\cdot 5 + \dotsb)w^2  + (2 \cdot 5 + \dotsb) w^3 + (5 + \dotsb ) w^4 + (1 + \dotsb ) w^5 + \dotsb
\end{equation*}
By examination, we have that $\mu(a_1) = 0$ and $a_1$ has five roots all lying in the region $v_5(-) = {1 \over 5}$. This example is even slightly worse than it appears because the zeroes of $a_1$ lie closer to the boundary than the arithmetic weights $z^k \chi_{25}$ corresponding to characters of conductor $25= 5^2$. A further computation, however, shows that this does not disprove Conjecture \ref{conj:in-text} for $p=5$ and $N=3$ with $r = 1$ even. Indeed, it follows from the data in Table \ref{table:p=5N=3} below that the index $i=1$ does not define a point on either the $w$-adic Newton polygon mod $5$ nor the Newton polygon of $P(\kappa,t)$ for any weight $\kappa$ with $0 < v_5(w(\kappa)) < 1$. Compare with the corresponding table on the component of $5$-adic weight space corresponding to weights $k \congruent 2 \bmod 4$ (see Table \ref{app:p=5N=3comp=2}).

\begin{table}[htdp]
\caption{$p=5, N=3$. Experimental observations for $$\det(1 - \restrict{tU_5}{S_{\kappa}^{\dagger}(3)}) = 1 + \sum a_i(w)t^i$$ on the component of $5$-adic weight space corresponding to weights $k \congruent 0 \bmod 4$. (The notation $\bf a_m$ means the value $a$ repeated $m$ times.)}
\begin{center}
\begin{tabular}{|c|c|c|l|}
\hline
$i$ & $\mu(a_i)$ & $\lambda(a_i)$ & Slopes of zeroes of $a_i$\\
\hline
1 & 0 & 5 & $\bf {1\over 5}_5$\\
2 & 0 & 2 & $\bf 1_2$\\
3 & 0 & 3 & $2,\bf 1_2$ \\
4 & 0 & 4 & $\bf 1_4$\\
5 & 0 & 7 & $\bf 1_7$\\
6 & 0 & 10 & $2,\bf 1_9$\\
7 & 0 & 14 & $\bf 2_2, \bf 1_{10} \bf {1\over 2}_2$\\
8 & 0 & 16 & $\bf 2_3, \bf 1_{13}$\\
\hline
\end{tabular}
\end{center}
\label{table:p=5N=3}
\end{table}
\end{example}

\begin{example}
Finally, the authors have not seen definitively different behavior in the location of the zeroes of the $a_i(w)$ in the so-called Buzzard irregular cases. For example, the prime $p=5$ is $\Gamma_0(14)$-irregular in the sense of \cite{Buzzard-SlopeQuestions}. The location of the zeroes of $a_i(w)$ for $i=1,\dotsc,8$ on the component $k\congruent 0 \bmod 4$ are recorded in Table \ref{table:p=5N=14} below. Here one observes that $a_6(w)$ and $a_7(w)$ have zeroes in the region $0 < v_5(w) < 1$, but that the sixth and seventh indices visibly do not lie on the $w$-adic Newton polygon of ${P(w,t)}$, and hence are irrelevant for the $w$-adic slopes. However, it is easily verified from the ultrametric inequality that the sixth and seventh indices are also irrelevant for $P(\kappa,t)$ for any weight $\kappa$ with $0 < v_5(w(\kappa)) < 1$ as well.
\begin{table}[htdp]
\caption{$p=5, N=14$. Experimental observations for $$\det(1 - \restrict{tU_5}{S_{\kappa}^{\dagger}(14)}) = 1 + \sum a_i(w)t^i$$ on the component of $5$-adic weight space corresponding to weights $k \congruent 0 \bmod 4$. (The notation $\bf a_m$ means the value $a$ repeated $m$ times.)}
\begin{center}
\begin{tabular}{|c|c|c|l|}
\hline
$i$ & $\mu(a_i)$ & $\lambda(a_i)$ & Slopes of zeroes of $a_i$\\
\hline
1 & 0 & 5 & $-$\\
2 & 0 & 2 & $-$\\
3 & 0 & 1 & $1$ \\
4 & 0 & 0 & $-$\\
5 & 0 & 1 & $1$\\
6 & 0 & 4 & $1, \bf {1\over 3}_3$\\
7 & 0 & 4 & $\bf 1_2, \bf {1\over 2}_{2}$\\
8 & 0 & 4 & $\bf 1_4,$\\
\hline
\end{tabular}
\end{center}
\label{table:p=5N=14}
\end{table}
\end{example}

\subsection{}
Here is a question which a computer could never answer.
\begin{question}
Is there an $r >0$ such that the locus of zeroes of $\set{a_i(w)}$ is uniformly bounded in $v_p(-) \geq r$?
\end{question}
Either a positive or a negative answer would be interesting, and a positive answer would prove Conjecture \ref{conj:in-text} (if $N=1$ at least).

\subsection{}
By Theorem A, if $N = 1$ then the $\mu$-invariants of every coefficient of the Fredholm series is zero. By Corollary A2, infinitely many $\mu$-invariants are zero for a general tame level.

\begin{question}
Can the $\mu$-invariants be positive for $N>1$?
\end{question}
\begin{answer}
Yes.
\end{answer}

\begin{example-special}\label{example:p2N3}
Let's consider $p=2$ and $N=3$. We will flesh out Section \ref{subsec:examplep=2N=3}.  By Theorem \ref{thm:koike-iwasawa} we have that
\begin{equation*}
T_1 = -c_3(2,1)[1] - {c_3(2,1,1) \over \rho_{1,1}^2-2}\left[{1 + \sqrt{-7}\over 2}\right]
\end{equation*}
where we choose the square root so that $\sqrt{-7} \congruent 1 \bmod 4$ in $\Z_2$. In the appendix we compute the constants and get $c_3(2,1) = 2$ (see Lemma \ref{lemma:const-term-special}) and $c_{3}(2,1,1)=0$ (see Lemma \ref{lemma:harder-constants}). Thus $T_1 = -2[1]$ meaning $a_1 = 2[1]$ has a positive $\mu$-invariant. And, in particular, we've shown that the function $\kappa \mapsto \tr(\restrict{U_2}{S_{\kappa}^{\dagger}(3)})$ is the constant function $-2$.
\end{example-special}
%
%

\subsection{}

Returning to Example \ref{example:p2N3}, can we say more about the $\mu$-invariants of the higher indices? 
\begin{proposition}
Let $P(w,t) = 1 + \sum a_i(w)t^i$ be the characteristic power series of $U_2$ acting on $2$-adic overconvergent cuspforms of level $\Gamma_0(3)$. Then $\mu(a_i) = 0$ if and only if $i$ is even.
\end{proposition}
\begin{proof}
Since $c_3(2,j) = 2$ uniformly in $j$ (see Lemma \ref{lemma:const-term-special}) it is not hard to see that for all $i$ the expression of $a_i$ as an element of $\Z_p[\![\Gamma]\!]$ is given by
\begin{equation*}
a_i = (i+1)[1] + \text{finite number of other terms}.
\end{equation*}
In particular, if $i$ is even then $a_i \not\congruent 0 \bmod 2$ by Lemma \ref{lemma:mu}.

Now we will show $a_i \congruent 0 \bmod 2$ when $i$ is odd. We already know that $a_1 = 2[1]$ by Example \ref{example:p2N3}. Suppose that $i > 1$ and by induction we suppose that $a_{i-j} \congruent 0 \bmod 2$ if $j < i$ and $j$ is even. Since $i$ is odd, the recursive formula
\begin{equation*}
i a_i = - \sum_{j=0}^i a_{i-j}T_j
\end{equation*}
implies that it suffices to show that the right hand side vanishes modulo 2. Then, by induction, it suffices to show that $T_j \congruent 0 \bmod 2$ if $j$ is odd. 

Consider the expression of $T_j$ as an Iwasawa function in  Theorem \ref{thm:koike-iwasawa}. We already noted that $c_3(2,1) = 2$ (see Lemma \ref{lemma:const-term-special}). On the other hand, if $j$ is odd then $\Delta_{s,j} = s^2 - 2^{j+2}\congruent s^2-2 \bmod 3$ and thus $3 \ndvd \Delta_{s,j}$ when $j$ is odd. It follows from Lemma \ref{lemma:harder-constants} that each of $c_3(2,s,j)$ (running over $1 \leq s < 2p^{j/2}$) is divisible by two as well and this completes the argument.
\end{proof}

\subsection{}
Inspired by Example \ref{example:p2N3}, Kevin Buzzard suggested the following question and answer.
\begin{question}
Can the $\mu$-invariants be arbitrarily large?
\end{question}
\begin{answer}
Yes.
\end{answer}
\begin{example}
Choose any sequence of primes $\ell_1,\ell_2,\dotsc$ with $\ell_i > 3$ for each $i$. Let $N_m = 3 \cdot \ell_1 \dotsb \ell_{m-1}$. Then $c_{N_m}(-) = c_{3}(-) \cdot \prod_{i=1}^{m-1} c_{\ell_i}(-)$. By  Lemma \ref{lemma:const-term-special} and Lemma \ref{lemma:harder-constants} we see that
\begin{align*}
c_{N_m}(2,1) &= 2^m,\\
c_{N_m}(2,1,1) &= 0.
\end{align*}
Thus, $\tr(\restrict{U_2}{S_{\kappa}^{\dagger}(N_{m})}) = -2^{m}[1]$ for any $m\geq 1$, and $\mu(a_1) = m$ for $p=2$ and tame level $N_m$.
\end{example}

The positive $\mu$-invariants are not a phenomena confined to the case $p=2$.
\begin{example}
Suppose that $N = \ell = 197$ and let $a_2$ be the second coefficient of the characteristic power series of $U_3$ acting on the $3$-adic overconvergent cuspforms of level $\Gamma_0(197)$. Then $a_1 = 2[1]$ and $a_2 = 3[1]$ are both constant and the latter has a positive $\mu$-invariant. 
\end{example}

\subsection{} In \cite{BuzzardKilford-2adc}, it is proven that if $p=2$ and $N=1$ then Conjecture \ref{conj:in-text} is true with $r=3$. The coefficients of the characteristic series in this case never have positive $\mu$-invariants and all their zeroes lie in the region $v_2(w)\geq 3$. In Table \ref{app:p=2N=3} below, we give the $\mu$ and $\lambda$-invariants for the coefficients $a_i(w)$ of the operator $U_2$ acting on $2$-adic overconvergent cuspforms of level $\Gamma_0(3)$. What one notices is the functions $a_i(w)$ have all their zeroes contained in the disc $v_2(w) \geq 3$, just like the Buzzard-Kilford example \cite{BuzzardKilford-2adc}.

\begin{question}
Is it possible that for $p=2$ Conjecture \ref{conj:in-text} is true with $r = 3$ for every $N\geq 1$?
\end{question}
\begin{answer}
No. If Conjecture \ref{conj:in-text} is true for $p=2$ and $N=3$ then $r \leq 2$ is necessary.
\end{answer}
\begin{example}\label{example:p=2N=3-bad-radius}
Let $p=2$ and $N=3$. It follows immediately from Table \ref{app:p=2N=3} that if $2 < v_2(w(\kappa)) < 3$ then the lowest slope in weight $\kappa$ is either 1 or repeated at least three times (it is likely 1). In particular, the first breakpoint occurs either at index 1 or at index greater than or equal to 3.  But, if $\chi_8$ is the even Dirichlet character of conductor $8$ then one can compute, in {\tt sage} for example, the slopes of $U_2$ acting on $S_4(\Gamma_1(24),\chi_8)$ to be ${1\over 2}, {1\over 2}, 1, 1, {3\over 2}, {3\over 2}, \dotsc$.  
\end{example}
We also checked that a similar phenomena occurs when $N=5$ (half the $\mu$-invariants are positive and the optimal radius for Conjecture \ref{conj:in-text} is at most 2). Lest the reader explain away these examples as being due to positive $\mu$-invariants, we also give the example of $N=7$.
\begin{example}\label{example:p2N7-break-BK}
Let $p=2$ and $N=7$. We computed the first nine coefficients of the characteristic power series and compiled the information into Table \ref{table:p2N7_fullspace}.
\begin{table}[htdp]
\caption{$p=2, N=7$. Experimental observations for $$\det(1 - \restrict{tU_2}{S_{\kappa}^{\dagger}(7)}) = 1 + \sum a_i(w)t^i$$ on the even component of $2$-adic weight space. (The notation $\bf a_m$ means the value $a$ repeated $m$ times.)} 
\begin{center}
\begin{tabular}{|c|c|c|l|}
\hline
$i$ & $\mu(a_i)$ & $\lambda(a_i)$ & Slopes of zeroes of $a_i$\\
\hline
1 & 0 & 0 & --\\
2 & $0$ & 2 & $\bf 1_2$\\
3 & $0$ & 2 & $\bf 1_2$\\
4 &$0$ & 2 & $\bf 4_2$ \\
5 &$0$ & 2 & $\bf 4_2$ \\
6 &$0$ & 6 & $\bf 4_2$, $3$, $2$, $\bf{1\over 2}_2$\\
7 &$0$ & 6 & $\bf{7\over 2}_2$, $\bf 3_2$, $\bf{1\over 2}_2$\\
8 &$0$ & 6 & $4$, $\bf{7 \over 2}_2$, $\bf 3_3$\\
9 &$0$ & 6 & $\bf 4_2$, $\bf 3_4$\\
\hline
\end{tabular}
\label{table:p2N7_fullspace}
\end{center}
\end{table}%

Consider the even Dirichlet character $\chi_8$ of conductor $8$. The classical space $S_2(\Gamma_1(56),\chi_8)$ of cuspforms is six dimensional. The characteristic polynomial of $U_2$ acting on this space (computed in {\tt sage}) is given by
\begin{equation*}
\det(\restrict{U_2-tI}{S_2(\Gamma_1(56),\chi_8)}) = t^6 + t^5 + 2t^4 + 4t^3 + 4t^2 + 4t + 8.
\end{equation*}
Thus the slopes of $U_2$ in classical weight two with character $\chi_8$ are given by $0, {1\over 2},{1\over 2}, {1\over 2}, {1\over 2}, 1$. It follows from Theorem \ref{thm:coleman-control} that the Newton polygon of $U_2$ acting on $S_{z^2\chi_8}^{\dagger}(7)$ must break at index $i=1$ and $i=5$.

Now suppose that Conjecture \ref{conj:in-text} is true for some $r > 2$. Since $v_2(w(z^2\chi_8)) = 1 < r$, the first two indices of break points of the Newton polygon(s) over the region $0 < v_2(-) < r$ are $1$ and $5$. But if $w_0$ is any weight with $2 < v_2(w_0) < r$, we check that there is a break point strictly between indices $1$ and $5$.  Indeed, it is immediate from Table \ref{table:p2N7_fullspace} that:
\begin{equation*}
{v_2(a_3(w_0)) - v_2(a_1(w_0)) \over 3-1} = {2 - 0 \over 2} = 1 
\end{equation*}
and
\begin{equation*}
{v_2(a_5(w_0))-v_2(a_3(w_0)) \over 5-3} = {2 v_2(w_0) -2  \over 2} = v_2(w_0) -1.
\end{equation*}
This is a contradiction since $v_2(w_0)-1 > 1$.
\end{example}
\begin{remark}
One cannot produce a concrete example of this phenomena by computing a classical space of modular forms since there is no classical weight in the region $2 < v_2(-) < 3$.
\end{remark}

\begin{appendix}
\renewcommand{\thesection}{\Roman{section}}

\section{Tables}\label{app:tables}
The tables below were constructed using an implementation of Koike's formula (see Theorem \ref{thm:koike-iwasawa}) in {\tt sage} \cite{sagemath}. The relevant programs are available on the websites of the authors.

We present the following data:\ for a fixed component of weight space (given by a congruence class $k \bmod p-1$) we give the $\mu$ and $\lambda$-invariants of the coefficients $a_i(w)$ of the characteristic power series, along with the valuations of the finitely many zeroes of each $a_i(w)$. The algorithm is slow, exponential in the maximal index $i$. Moreover, as $i\goto \infty$, one must use higher and higher $(w,p)$-adic precision. 

For example, if one wants to reproduce results in Table \ref{p=2N=1} ($p=2$ and $N=1$) up to $i=10$, it is enough to only work modulo $(w^{56}, 2^{250})$.  The $w$-adic precision can be computed ahead of time by \cite{BuzzardKilford-2adc} and the necessary $2$-adic precision can be estimated by \cite[Corollary 1]{BuzzardCalegari-2adicSlopes}. This takes roughly a minute on the first author's personal laptop. In contrast, going up to $i=20$ requires roughly precision $(w^{211}, 2^{850})$ and took almost eight hours to complete.

Throughout, the notation $\bf a_m$ means the value $a$ repeated $m$ times.

\begin{table}[!htdp]
\caption{$p=2, N = 1$. Experimental observations for $$\det(1 - \restrict{tU_2}{S_{\kappa}^{\dagger}(1)}) = 1 + \sum a_i(w)t^i$$ on the even component of $2$-adic weight space.}
\begin{center}
\begin{tabular}{|c|c|c|l|}
\hline
$i$ & $\mu(a_i)$ & $\lambda(a_i)$ & Slopes of zeroes of $a_i$\\
\hline
1 & 0 & 1 & 3\\
2 & 0 & 3 & 4, $\bf 3_2$\\
3 & 0 & 6 & 7, 4, $\bf 3_4$ \\
4 & 0 & 10 & 7, 5, $\bf 4_2$, $\bf 3_6$\\
5 & 0 & 15 & $6,{\bf 5_2}, {\bf 4_3}, {\bf 3_9}$\\
6 & 0 & 21 & $8, 6, {\bf 5_2}, {\bf 4_5}, {\bf 3_{12}}$\\
7 & 0 & 28 & $\bf 8_2, 6, \bf 5_3, \bf 4_6, \bf 3_{16}$\\
8 & 0 & 36 & $\bf 8_2, \bf 6_2, \bf 5_4, \bf 4_8, \bf 3_{20}$\\
9 & 0 & 45 & $8, 7, \bf 6_3, \bf 5_5, \bf 4_{10}, \bf 3_{25}$\\
10 & 0 & 55 & $\bf 7_2, \bf 6_4, \bf 5_6, \bf 4_{13}, \bf 3_{30}$ \\
$\vdots$ & $\vdots$ & $\vdots$ & $\vdots$\\
20 & 0 & 210 & $9, \bf 8_4, \bf 7_8, \bf 6_{12}, \bf 5_{25}, \bf 4_{50}, \bf 3_{110}$\\
\hline
\end{tabular}
\end{center}
\label{p=2N=1}
\end{table}%

\begin{table}[!htdp]
\caption{$p=3, N = 1$. Experimental observations for $$\det(1 - \restrict{tU_3}{S_{\kappa}^{\dagger}(1)}) = 1 + \sum a_i(w)t^i$$ on the  component of $3$-adic weight space corresponding to $k\congruent 0 \bmod 2$.}
\begin{center}
\begin{tabular}{|c|c|c|l|}
\hline
$i$ & $\mu(a_i)$ & $\lambda(a_i)$ & Slopes of zeroes of $a_i$\\
\hline
1 & 0 & 2 & $\bf 1_2$\\
2 & 0 & 6 & $3, \bf 1_5$\\
3 & 0 & 12 & $3, \bf 2_2, \bf 1_9$\\
4 & 0 & 20 & $3, \bf 2_4, \bf 1_{15}$\\
5 & 0 & 30 & $4, \bf 3_2, \bf 2_5, \bf 1_{22}$\\
6 & 0 & 42 & $\bf 4_2, \bf 3_2, \bf 2_8, \bf 1_{30}$\\
\hline
\end{tabular}
\end{center}
\label{p=3N=1}
\end{table}%

\begin{table}[!htdp]
\caption{$p=2, N=3$. Experimental observations for $$\det(1 - \restrict{tU_2}{S_{\kappa}^{\dagger}(3)}) = 1 + \sum a_i(w)t^i$$ on the even component of $2$-adic weight space.}
\begin{center}
\begin{tabular}{|c|c|c|l|}
\hline
$i$ & $\mu(a_i)$ & $\lambda(a_i)$ & Slopes of zeroes of $a_i$\\
\hline
1 & 1 & 0 & $-$\\
2 & 0 & 1 & 4\\
3 & 1 & 2 & $\bf 3_2$ \\
4 & 0 & 3 & $4, \bf 3_2$\\
5 & 1 & 4 & $8, 4, {\bf 3_2}$\\
6 & 0 & 6 & $6,5,4,{\bf 3_2}$\\
7 & 1 & 8 & $6, {\bf 4_3}, {\bf 3_4}$\\
\hline
\end{tabular}
\end{center}
\label{app:p=2N=3}
\end{table}

\begin{table}[!htdp]
\caption{$p=5, N=3$. Experimental observations for $$\det(1 - \restrict{tU_5}{S_{\kappa}^{\dagger}(3)}) = 1 + \sum a_i(w)t^i$$ on the  component of $5$-adic weight space corresponding to weights $k \congruent 2 \bmod 4$.}
\begin{center}
\begin{tabular}{|c|c|c|l|}
\hline
$i$ & $\mu(a_i)$ & $\lambda(a_i)$ & Slopes of zeroes of $a_i$\\
\hline
1 & 0 & 0 & $-$\\
2 & 0 & 1 & 1\\
3 & 0 & 3 & $\bf 2_2, 1$ \\
4 & 0 & 5 & $\bf 2_2, \bf 1_3$\\
5 & 0 & 7 & $\bf 2_3 \bf 1_4$\\
6 & 0 & 9 & $\bf 2_4, \bf 1_5$\\
7 & 0 & 12 & $\bf 2_5, \bf 1_7$\\
8 & 0 & 16 & $\bf 2_5, \bf 1_{11}$\\
\hline
\end{tabular}
\end{center}
\label{app:p=5N=3comp=2}
\end{table}

\pagebreak

\section{Constants in trace formulae}
\label{app:constants}

In this appendix, we briefly record information about the constants that appear in Corollary \ref{thm:koike-iwasawa} for a general tame level $N$. We refer to the notation of Hijikata, especially \cite[pp.\ 57--58]{Hijikata-TraceFormula}. Among all the terms in \cite[Theorem 0.1]{Hijikata-TraceFormula}, the only ones which survive the $p$-adic limit in Corollary \ref{thm:koike-iwasawa} are the terms labeled (h) and (e). Among those terms, (h) gives rise to the ``constant term'' $c_N(p,j)$ in front of the group-like element $[1]$ in Corollary \ref{thm:koike-iwasawa} and the term(s) labeled (e) gives rise to the other constants.

\subsection{}
It's not hard to see that the integer $s = p^j+1$ is the unique positive integer $s$ co-prime to $p$ such that $s^2 - 4p^j = t^2$ for an integer $t$, in which case we have $t = p^j -1$. Then the definition of $c_N(p,j)$ in \cite{Hijikata-TraceFormula} is
\begin{equation*}
c_N(p,j) := {1 \over p^j-1}\sum_{f \dvd (p^j-1)} \phi\left(p^j-1\over f\right) c(p^j+1,f),
\end{equation*}
where $c(s,f)$ is as explained on \cite[p.\ 58]{Hijikata-TraceFormula}. Simple examples when $N$ is not square-free show that you really have to compute the sum. On the other hand, an easy special case is the following.
\begin{lemma}\label{lemma:const-term-special}
If $N = \ell$ is a prime different from $p$ then $c(p^j+1,f) = 2$ for all $f \dvd p^j-1$ and thus $c_{\ell}(p,j) = 2$ for all primes $\ell \neq p$.
\end{lemma}
\begin{proof}
We only sketch the computation since it is easily verified. If $p^j - 1 \not\congruent 0 \bmod \ell$ then $x^2 - (p^j+1)x + p^j$ has exactly two roots modulo $\ell^m$ for all $m\geq 1$, namely $p^j$ and $1$. Thus, to compute $c(p^j+1,f)$ in this case, one only uses the $A$-terms in \cite[p.\ 58]{Hijikata-TraceFormula} and we see $c(p^j+1,f) = 2$. If $p^j-1 \congruent 0 \bmod \ell$ then it could be more complicated except that in the end the sets defined $A_{\eta}(\dotsb)$ and $B_{\eta}(\dotsb)$ in \cite{Hijikata-TraceFormula} are counting certain solutions to $x^2 - (p^j+1)x + p^j$ modulo $\ell^{v_\ell(N)} = \ell$, which in this case is $(x-1)^2$. So the computation is easy still.
\end{proof}

\begin{remark}
If $N$ and $M$ are coprime then $c_{NM}(p,j) = c_N(p,j)c_M(p,j)$ and so Lemma \ref{lemma:const-term-special} also gives a formula for square-free $N$.
\end{remark}

\subsection{}
The rest of the terms in Theorem \ref{thm:koike-iwasawa} are computed  from the terms denoted by (e) in \cite{Hijikata-TraceFormula}. For each $s$ such that $p\ndvd s$ and $1\leq s < 2p^{j/2}$, write $\Delta_{s,j}=s^2-4p^j = t^2 D_K$ where $D_K$ is the discriminant of an imaginary quadratic field. Then 
\begin{equation}\label{eqn:funny-constant}
c_N(p,s,j) := \sum_{f \dvd t} \hat h\left({s^2 - 4p^j\over f^2}\right) c(s,f)
\end{equation}
where $c(s,f)$ is as on \cite[p.\ 58]{Hijikata-TraceFormula} again, and $\hat h(D)$ is defined as a certain normalized class number\footnote{If we write $D = t^2D_K$ where $D_K$ is the discriminant of an imaginary quadratic field, then $\cal O_D = \Z + t\cal O_K$ is an order in $K$ and $\hat h(D) := 2h(\cal O_D)/w(\cal O_D)$ where $h(\cal O_D)$ is the usual class number and $w(\cal O_D)$ is the cardinality of the roots of unity.}. It can be computed from tables of class numbers of imaginary quadratic fields (e.g.\ on a computer) by the formula
\begin{equation*}
\hat h(D) = t \cdot \hat h(D_K) \prod_{\ell \dvd t} \left(1 - \left(D_K \over \ell\right){1\over \ell}\right)
\end{equation*}
where $D = t^2 D_K$ and $D_K$ is a fundamental discriminant. One can check that this agrees with the formulae given for $N=1$ in Section \ref{subsec:koikeN=1} by using that $c(s,f) = 1$ for all $s$ and $f$ if $N=1$ and the well-known identity for the Hurwitz class numbers
\begin{equation}\label{eqn:hurwitz-to-hath}
H(4p^j-s^2) = \sum_{f \dvd t}  \hat h\left({s^2-4p^j \over f^2}\right).
\end{equation}
We make a similar computation as Lemma \ref{lemma:const-term-special} in the case where $N =\ell$ is prime.

\begin{lemma}\label{lemma:harder-constants}
Suppose that $N = \ell$ is a prime different from $p$ and let $\Delta_{s,j} = s^2 - 4p^j$. If $(\Delta_{s,j},\ell) = 1$ then
\begin{equation*}
c(s,f) = \begin{cases}
2 & \text{if $\left({\Delta_{s,j} \over \ell}\right) = 1$}\\
0 & \text{if $\left({\Delta_{s,j} \over \ell}\right) =-1$}
\end{cases}
\end{equation*}
where $\left({\ast \over\ast }\right)$ is the Kronecker symbol. In particular, $c(s,f)$ is independent of $f$ and
\begin{equation*}
c_\ell(p,s,j) = \begin{cases}
2\cdot H(4p^j-s^2) & \text{if $\left(\Delta_{s,j} \over \ell\right)=1$}\\
0 & \text{if $\left(\Delta_{s,j} \over \ell\right)=-1$}.
\end{cases}
\end{equation*}
\end{lemma}
\begin{proof}
We omit the proof of the first part as it follows easily from the formulae given in \cite{Hijikata-TraceFormula}, as in Lemma \ref{lemma:const-term-special}. The second part of the statement follows the first part, the definition \eqref{eqn:funny-constant} and the identity \eqref{eqn:hurwitz-to-hath}.
\end{proof}

\end{appendix}

\bibliography{slope_bib}
\bibliographystyle{abbrv}

\end{document}